\documentclass[11pt]{article}
\usepackage[english]{babel}
\usepackage{amsmath}
\usepackage{amsthm}
\usepackage{amssymb}

\usepackage{graphicx}
\usepackage{epsfig}
\usepackage{eufrak}
\usepackage{mathrsfs}
\usepackage{enumerate}
\usepackage{scalerel}

\usepackage[dvips]{color}
\newtheorem{theorem}{Theorem}[section]
\newtheorem{definition}{Definition}[section]

\newtheorem{lemma}{Lemma}[section]
\newtheorem{proposition}{Proposition}[section]

\newtheorem{remark}{Remark}[section]
\newtheorem{example}{Example}[section]

\def\C{\mathbb C}

\def\R{\mathbb R}
\def\F{\mathbb F}

\def\N{\mathbb N}

\def\Z{\mathbb Z}
\def\T{\mathbb T}

\def\D{\mathbb D}

\def\bff{{\boldsymbol f}}

\def\({{\rm (}}
\def\){{\rm )}}

\title{\bf Superoscillating sequences and
supershifts for families of generalized functions}

\author{F. Colombo\footnote{Politecnico di
Milano, Dipartimento di Matematica, Via E. Bonardi, 9 20133 Milano,
Italy},\  I. Sabadini$^*$,\ D.C. Struppa\footnote{The Donald Bren Presidential Chair in Mathematics, Chapman University, Orange, USA},
 A. Yger\footnote{IMB, Universit\'e de Bordeaux, 33405, Talence, France}}

\begin{document}
\numberwithin{equation}{section}
\maketitle

\begin{abstract}
We construct in this paper a large class of superoscillating sequences, more generally of $\mathscr F$-supershifts, where $\mathscr F$ is a family of smooth functions (resp. distributions, hyperfunctions) indexed by a real parameter $\lambda\in \R$.
The key model we introduce in order to generate such families is the evolution through a Schr\"odinger equation $(i\partial/\partial t - \mathscr H(x))(\psi)=0$ with a suitable hamiltonian $\mathscr H$, in particular a suitable potential $V$ when $\mathscr H(x) = -(\partial^2/\partial x^2)/2 + V(x)$. The family $\mathscr F$ is in this case $\mathscr F= \{(t,x) \mapsto \varphi_\lambda(t,x)\,;\, \lambda \in \R\}$, where $\varphi_\lambda$ is evolved from the initial datum $x\mapsto e^{i\lambda x}$. Then $\mathscr F$-supershifts will be of the form $\{\sum_{j=0}^N C_j(N,a) \varphi_{1-2j/N}\}_{N\geq 1}$ for $a\in \R\setminus [-1,1]$, taking $C_j(N,a) =\binom{N}{j}(1+a)^{N-j}(1-a)^j/2^N$. We prove the locally uniform convergence of derivatives of the supershift towards corresponding derivatives of its limit.
We analyse in particular the case of the quantum harmonic oscillator, which forces us, in order to take into account singularities of the evolved datum, to enlarge the notion of supershifts for families of functions to a similar notion for families of hyperfunctions, thus beyond the frame of distributions.
\end{abstract}


\section{Introduction}
The Aharonov-Berry superoscillations are  band-limited functions that  can oscillate faster than
their fastest Fourier component.
These functions (or sequences) appear in the study of Aharonov weak measurements, \cite{aav,abook,av,b5}.
The literature related to superoscillations is very large; without claiming completeness, we mention
 \cite{berry,berry-noise-2013,berry2,b1,b4,kempf2,leeferreira,kempf1,lindberg}.
Quite recently, this class of functions has been investigated from the mathematical point of view, see \cite{acsst4,acsst3,acsst1,acsst6,JFAA, AOKI,QS2, BCS,
harmonic,CSY,CSSY} and the monograph \cite{acsst5}. Their theory is now very well
developed, even though  there are still open problems associated with
 superoscillatory functions, in particular as it concerns their longevity, when evolved according to a wide class of partial differential equations.
\vskip 2mm
\noindent
Let $a>1$ be a real number. The archetypical superoscillatory sequence
is the sequence of complex valued functions $\{x\mapsto F_N(x,a)\}_{N\geq 1}$ defined on $\mathbb{R}$ by
\begin{equation}\label{sect1-eq1}
F_N(x,a)=\Big(\cos\Big(\frac{x}{N}\Big)+ia\sin \Big(\frac{x}{N}\Big) \Big)^N
=\sum_{j=0}^N C_j(N,a)e^{i(1-2j/N){x}}
\end{equation}
where
\begin{equation}\label{sect1-eq2}
C_j(N,a)={N\choose j}\left(\frac{1+a}{2}\right)^{N-j}\left(\frac{1-a}{2}\right)^j,
\end{equation}
and ${N\choose j}$ denotes the binomial coefficient.
The first thing one notices is that if we fix $x \in \mathbb{R}$, and we let $N$ go to infinity, we immediately obtain that
$$
\lim_{N \to \infty} F_N(x,a)=e^{iax}.
$$
This new representation, together with the computation of the limit of $x\mapsto F_N(x,a)$ when $N$ goes to infinity,
explains why such a sequence $\{x\mapsto F_N(x,a)\}_{N\geq 1}$ is called {\it superoscillatory}. Even though, for every $N$,
the frequencies $(1-2j/N)$ that appear in the Fourier representation of $F_N$ are bounded by one, the limit function is $x\mapsto e^{iax}$, where $a$ can be an arbitrarily large real number. If one considers the map
$\lambda \in \R \mapsto \varphi_\lambda$, where $\varphi_\lambda~: x\in \R \mapsto e^{i\lambda x}$, one says that
$\lambda \mapsto \{x \mapsto F_N(x,\lambda)\}_{N\geq 1}$ realizes a
{\it supershift} for $\lambda \mapsto \varphi_\lambda$, or also that
$\lambda \mapsto \varphi_\lambda$ admits $\lambda \mapsto
\{x\mapsto F_N(x,\lambda)\}_{N\geq 1}$ as a {\it supershift}.
Such a notion will be made precise in Definition \ref{sect5-def1}.
\vskip 2mm
\noindent
Let $t\in \R$ or $t\in \R^+$ be a real parameter and $P= \gamma_0 + \gamma_1 X  + \cdots + \gamma_d X^d \in \C[X]$ with $\gamma_d\not=0$. For any $\lambda \in \R$ and $N\in \N^*$, let
$$
\psi_{P,N}(t,x,\lambda) = \sum\limits_{j=0}^N C_j(N,\lambda)\, e^{i(1-2j/N)x} e^{it P(1-2j/N)},
$$
so that
\begin{equation*}
\Big[i\, \frac{\partial}{\partial t} +
\Big(\sum\limits_{\kappa =0}^d  \gamma_\kappa \Big(-i \frac{\partial}{\partial x}\Big)^\kappa\Big)\Big] (\psi_{P,N}(t,x,\lambda)) \equiv 0\quad {\rm on}\ \R^2_{t,x},
\end{equation*}
together with the initial condition
\begin{equation*}
\big[\psi_{P,N}(t,x,\lambda)\big]_{t=0}  = F_N(x,\lambda).
\end{equation*}
In the particular case where $P$ is even with real coefficients, namely $P = \sum_{\kappa'=0}^{d'} \gamma_{2\kappa'} X^{2\kappa'}$ with
$\gamma_{2\kappa'}\in \R$ for $\kappa'=0,...,d'$, $\gamma_{2d'}\neq 0$ and $\check P= \sum_{\kappa'=0}^{d'} (-1)^{\kappa'+1}\gamma_{2\kappa'} X^{2\kappa'}$, the function $(t,x)\in \R^2_{t,x}\longmapsto \psi_{P,N}(a,t,x)$ is the  global solution of the Schr\"odinger type partial differential equation
$\big(i\partial/\partial t - \check P(\partial/\partial x)\big)(\psi)\equiv 0$ in $\R^2_{t,x}$  evolved from the initial datum
$x\mapsto F_N(x,\lambda)$ on the line $\{t=0\}$ in $\R^2_{t,x}$.
\vskip 2mm
\noindent
Let $D_x := \partial/\partial x$ and $\odot$ denote the composition law between differential operators in the variable $x$ with coefficients depending on the parameter $t$. Observe then that one can rewrite formally for any $N\in \N^*$
\begin{multline*}
\psi_{P,N}(t,x,\lambda) = \sum\limits_{j=0}^N C_j(N,\lambda)\,
e^{ix(1-2j/N)}\, \prod_{\kappa =0}^d
\Big(\sum\limits_{\ell =0}^\infty
\frac{(i^{1-\kappa}\, t \gamma_\kappa)^\ell}{\ell!}
(i(1-2j/N))^{\kappa \ell}\Big) \\
= \sum\limits_{j=0}^N
C_j(N,\lambda) \, \Big(\bigodot_{\kappa = 0}^d\big(
\sum\limits_{\ell =0}^\infty
\frac{(i^{1-\kappa}\, t \gamma_\kappa)^\ell}{\ell!}\, D_x^{\kappa\,\ell}\big)\Big)\, (e^{ix(1-2j/N)}) \\
=
\Big(\bigodot_{\kappa = 0}^d
\big(\sum\limits_{\ell =0}^\infty
\frac{(i^{1-\kappa}\, t \gamma_\kappa)^\ell}{\ell!}\, D_x^{\kappa\,\ell}\big)\Big) \big(\sum\limits_{j=0}^N C_j(N,\lambda)\, e^{ix(1-2j/N)}\big).
\end{multline*}
We will justify such a rewriting in section \ref{sect3} and exploit it
in order to prove (theorem \ref{sect3-thm3}) that for any
$(\mu,\nu)\in \N^2$
$$
\frac{\partial^{\mu+\nu}}{\partial t^\mu
\partial x^\nu}
(\psi_{P,N}(t,x,\lambda)) \stackrel{N\longrightarrow \infty}{\longrightarrow}
\frac{\partial^{\mu+\nu}}{\partial t^\mu
\partial x^\nu} \, (e^{itP(\lambda)}\, e^{i\lambda x})
$$
Let $a\in \R \setminus [-1,1]$. We deduce in this way from the original ``superoscillating'' sequence
$\{x \mapsto F_N(x,a)\}_{N\geq 1}$ a large class of superoscillating sequences which are of the form $\{x\mapsto \psi_{P,N}(t,x,a)\}_{N\geq 1}$ where $t\in \R$ is interpreted as a real parameter, the superoscillating convergence being uniform with respect to the parameter $t$~; moreover the superoscillating convergence property propagates through any differential operator in $t,x$, the convergence being uniform on any compact subset of $\R^2_{t,x}$.
\vskip 2mm
\noindent
As we already mentioned, the class of superoscillating sequences
$\{x\mapsto \psi_{P,N}(t,x,a)\}_{N\geq 1}$ introduced previously includes examples where $(t,x)\mapsto \psi_{P,N}(t,x,a)$ is realized through the evolution in $t$ (from the initial value $t=0$) of the solution of a Cauchy problem (with initial datum on $\{t=0\}$) attached to a Schr\"odinger operator $i\partial/\partial t - \check P(\partial/\partial x)$, where $P$ is a real even differential operator with constant coefficients (the case where $\check P(\partial/\partial x) = - \partial^2/\partial x^2$ corresponds for example to the classical case of Schr\"odinger equation for a free particle).
Indeed if, for $\lambda\in \R$, one denotes
 the response (at the time $t$) to the input datum $x\mapsto e^{i\lambda x}$ (when $t=0$) with the function $(t,x)\longmapsto \varphi_\lambda(t,x)$, then for any $a\in \R$ the function $(t,x)\mapsto \psi_{P,N}(t,x,a)$ can be expressed as
$$
\psi_{P,N}(t,x,a) = \sum\limits_{j=0}^N C_j(N,a)\, \varphi_{1-2j/N}(t,x).
$$
Since, when $\lambda>1$ is arbitrarily large and $N\in \N^*$, the function $(t,x)\mapsto \psi_{P,N}(t,x,\lambda)$ arises from {\it shifted} functions
$(t,x)\mapsto \varphi_\lambda(t,x)$ with $|\lambda|\leq 1$, one can still say that $\lambda \in \R \mapsto \varphi_\lambda$ (considered from $\R$ into the space of functions of the variables $(t,x)$ in the phase domain, here $\R^2_{t,x}$)
admits $\lambda \mapsto \big\{(t,x)\mapsto \psi_{P,N}(t,x,\lambda)\big\}_{N\geq 1}$ as a {\it supershift}.
\vskip 2mm
\noindent
Given a Schr\"odinger operator
$$
i\frac{\partial}{\partial t} + \frac{1}{2} \frac{\partial^2}{\partial x^2} + V(x)
$$
with a suitable real potential $V$ and Green function $G_V~: (t,x,x') \mapsto
G_V(t,x,0,x')$ such that
$$
\varphi_\lambda (t,x) = \int_\R G_V(t,x,0,x')\, e^{i\lambda x'}\, dx' \\\quad {\rm or}
\int_{\R^+} G_V(t,x,0,x')\, e^{i\lambda x'}\, dx'
$$
for $\lambda \in \R$ and $(t,x)$ in the phase domain (as a regularized integral on $\R$ or $\R^+$ in a sense that will be precised in section \ref{sect4}), we will settle sufficient conditions that ensure in particular that the integral operator
\begin{multline*}
T(x'\mapsto f(x'))(t,x)  = \int_{\R} G_V(t,x,0,x')\, (x'\mapsto f(x'))\, dx'\\{\rm or}\quad
\int_{\R^+} G_V(t,x,0,x')\, (x'\mapsto f(x'))\, dx'
\end{multline*}
is such that for any $\lambda \in \R$,
\begin{multline*}
T \Big(x' \mapsto \sum\limits_{j=0}^N C_j(N,\lambda)\, e^{i(1-2j/N) x'}\Big)(t,x)
= \sum\limits_{j=0}^N
C_j(N,\lambda)\, \varphi_{1-2j/N}(t,x)\\
 \stackrel{N\longrightarrow \infty}{\longrightarrow} \, T(x'\mapsto e^{i\lambda x'})(t,x) =
 \varphi_\lambda(t,x)
\end{multline*}
locally uniformly in some open subset $\mathscr U$ of the phase space (in $\R^2_{t,x}$) on which $V$ is smooth and which
is entirely determined by the explicit expression of the Green function $G_V$ (theorem \ref{sect5-thm1}).
In such a situation the map $\lambda \in \R \mapsto (\varphi_\lambda)_{|\mathscr U}$ admits then the sequence
\begin{equation*}
\big\{\lambda \mapsto \big(\sum\limits_{j=0}^N C_j(N,a)\, T(
x'\mapsto e^{i(1-2j/N)x'})\big)_{|\mathscr U}\big\}_{N\geq 1} =
\big\{\lambda \mapsto \big(\sum\limits_{j=0}^N C_j(N,a)\, \varphi_{1-2j/N}\big)_{|\mathscr U}\big\}_{N\geq 1}
\end{equation*}
(where $a\in \R\setminus [-1,1]$)
as a supershift. Moreover, for any $\lambda \in \R$,
$\varphi_\lambda \in \mathscr C^\infty(\mathscr U,\C)$ and for any
$(\mu,\nu)\in \N^2$,
\begin{equation*}
\frac{\partial^{\mu+\nu}}{\partial t^\mu
\partial x^\nu} \Big(
\sum\limits_{j=0}^N C_j(N,a) \, \varphi_{1-2j/N}\Big)
\stackrel{N\longrightarrow \infty}{\longrightarrow}
\frac{\partial^{\mu+\nu}}{\partial t^\mu
\partial x^\nu} \, (\varphi_a)
\end{equation*}
locally uniformly in $\mathscr U$.
\vskip 2mm
\noindent
Interesting situations occur when $\lambda\in \R \mapsto
\varphi_\lambda$ makes sense as a continuous map from
$\R$ into $\mathscr D'(\mathscr U',\C)$ for some open subset $\mathscr U'
\supsetneq \mathscr U$ in the phase space. Such is the case in the example of the quantum harmonic oscillator, where
$V(x)=x^2/2$, the phase space is $\R^{+*} \times \R$ and
$$
\mathscr U = (\R^+ \times \R) \setminus \{\big(k\, \pi/2,x\big)\,;\,
k\in \N,\ x\in \R\} \subset \mathscr U' = \R^{+*} \times \R.
$$
In such case, given $k'\in \N$ and $x_0\in \R$, it is impossible to interpret
\begin{equation}\label{sect1-eq3}
\big\{\lambda \mapsto \big(\sum\limits_{j=0}^N C_j(N,a)\, \varphi_{1-2j/N}\big)_{{\rm about}\ ((2k'+1)\pi/2,x_0)}\big\}_{N\geq 1}
\end{equation}
(when $a\in \R \setminus [-1,1]$) as a supershift for $\lambda \longmapsto (\varphi_\lambda)_{|{\rm about}\, ((2k'+1)\pi/2,x_0)}$ (all maps being considered here as
distribu\-tion-valued about $((2k'+1)\pi/2,x_0)$), while it is possible to do so about a point $(k''\pi,x_0)$, where $k''\in \N^*$.
In order to interpret \eqref{sect1-eq3} as a supershift for $\lambda \mapsto (\varphi_\lambda)_{{\rm about}\ ((2k'+1)\pi/2,x_0)}$,
one needs to consider $(\varphi_\lambda)_{|{\rm about}\ ((2k'+1)\pi/2,x_0)}$ as a hyperfunction (in $t$) times a distribution (in $x$) instead of distribution in $(t,x)$. We will discuss such questions in section \ref{sect6}.

The plan of the paper is the following: the paper contains five sections, besides this introduction. In section 2 we introduce the spaces $A_p(\mathbb C)$, $A_{p,0}(\mathbb C)$, and we define some infinite order differential operators with nonconstant coefficients which will play a crucial role to prove our main results. In section 3  we recall the definition of generalized Fourier sequence and (complex) superoscillating sequence in one and several variables together with some examples; we then study two Cauchy-Kowalevski problems (one of which of Schr\"odinger type) and we show that superoscillations persist in time. In section 4 we address the problem of explaining the process of regularization of formal Fresnel-type integrals which is a necessary step to obtain further results in the paper. Fresnel-type integrals are shown to be continuous on $A_1(\mathbb C)$ in section 5, in which we also treat a Cauchy problem for the Schr\"odinger equation with centrifugal potential and also for the quantum harmonic oscillator. Finally, in section 6, we investigate the evolution of superoscillating initial data with
respect to the notion of supershift for the quantum harmonic oscillator, and we focus
 on singularities. It is interesting to note that in this case one needs to extend the concept of supershift in the case of hyperfunctions.
\vskip 1mm
\noindent
{\bf Notations.} We use the notations with capital letters $Z,d/dZ,W,d/dW,\check Z$ in the expressions of formal differential operators, besides the usual notation $z$ for the complex variable and $t$ (time)  $x,x'$ (space) real variables.

\section{On continuity of some convolution operators}

Let $f$  be a non-constant entire function of a complex variable $z$. We define
$$
M_f(r)=\max_{|z|=r}|f(z)|,\ \ \ \text{ for}\ \ r\geq 0.
$$
The non-negative real number $\rho$ defined by
$$
\rho=\limsup_{r\to\infty}\frac{\ln\ln M_f(r)}{\ln r}
$$
is called the {\it order} of $f$. If $\rho$ is finite then $f$ is said to be of {\it finite order} and if $\rho=\infty$ the function $f$ is said to be of {\it infinite order}.
\vskip 2mm
\noindent
In the case $f$ is of finite order we introduce the non-negative real number
$$
\sigma=\limsup_{r\to\infty}\frac{\ln M_f(r)}{ r^\rho},
$$
and call it the {\it type} of $f$. If $\sigma\in (0,\infty)$ we say
$f$ is {\it of normal type}, while we say it is {\it of minimal type} if $\sigma=0$ and {\it of maximal type} if $\sigma=\infty$.
Constant entire functions are considered of minimal type and order zero. In the sequel we will extensively make use of {\it weighted spaces}
$A_p(\C)$ or $A_{p,0}(\C)$ of entire functions whose definition follows~; such spaces are classical, see e.g. \cite{BG_book,Taylor}. 

\begin{definition}\label{sect2-def1}
Let $p$ be a strictly positive number.
We define the space $A_p(\C)$ as the $\C$-algebra of entire functions such that there exists $B>0$ such that
$$
\sup\limits_{z\in \C} \big(|f(z)|\, \exp(-B|z|^p\big)  <+\infty.
$$
The space $A_{p,0}(\C)$ consists of those entire functions such that
$$
\forall\, \varepsilon>0,\ \sup\limits_{z\in \C}
\big(|f(z)|\, \exp(-\varepsilon|z|^p) < +\infty.
$$
\end{definition}
\vskip 2mm
\noindent
To define a topology in these spaces we follow \cite[Section 2.1]{BG_book}. For $p>0$, $B>0$ and for any entire function $f$, we set
\[
\|f\|_B:=\sup_{z\in{\mathbb C}}\{|f(z)|\exp(-B|z|^p)\}.
\]
Let $A_p^B(\C)$ denote the $\C$-vector space of entire functions satisfying $\|f\|_B<\infty$. Then $\|\cdot\|_B$ defines a norm on
$A_p^B(\C)$ so that $(A_p^B(\C),\|\ \|_B)$ is a Banach space and the natural inclusion mapping $A_p^B\hookrightarrow A_p^{B'}$
(when $0<B\leq B'$) is a compact operator from $(A_p^B(\C),\|\ \|_B)$ into
$(A_p^B(\C),\|\ \|_{B'})$.
For any sequence $\{B_n\}_{n\geq 1}$ of positive numbers, strictly increasing to infinity, we can introduce an LF-topology on $A_p(\C)$ given by the inductive limit
\[
A_p(\C):=\lim_{\longrightarrow} A_p^{B_n}(\C) .
\]
Since this topology is stronger than the topology of the pointwise convergence, it is independent of the choice of the sequence $\{B_n\}_{n\geq 1}$. Thus, in this inductive limit topology,
given $f$ and a sequence $\{f_N\}_{N\geq 1}$ in $A_p(\C)$, we say that  $f_N\to f$ in $A_p(\C)$ if and only if there exists $n\in \N^*$ such  that $f, f_N\in A_p^{B_n}(\C)$ for all $N\in \N^*$,  and $\|f_N-f\|_{B_n}\to 0$ for $N\to\infty$. The topology on $A_{p,0}(\C)$ is given as the projective limit
\[
A_{p,0}(\C):=\lim_{\longleftarrow} A_p^{\varepsilon_n}(\C)
\]
where $\{\varepsilon_n\}_{n\geq 1}$ is a strictly decreasing sequence of positive numbers converging to $0$.
It can be proved, see \cite[Section 6.1]{BG_book}, that $A_p(\C)$  and $A_{p,0}(\C)$ are respectively a DFS space and an FS space.
When $p>1$, $A_{p,0}(\C)$ is the strong dual of $A_{p'}(\C)$
(where $1/p+1/p'=1$), the duality being realized as
$$
\mu \in (A_{p'}(\C))' \longmapsto
\Big[{\rm Fourier-Borel\ Transform\ of}\ \mu~: w\in \C \longmapsto \mu_z(e^{-zw})\Big] \in A_{p,0}(\C).
$$
In the extreme case $p=1$, $A_1(\C)$ (also denoted as ${\rm Exp}(\C)$) is isomorphic to the space $\widehat{H(\C)}$ of analytic functionals, the duality being realized as
$$
T \in \widehat{H(\C)} \longmapsto \Big[
{\rm Fourier-Borel\ Transform\ of}\ T~: w\in \C \mapsto
T_z(e^{-zw})\Big] \in A_1(\C).
$$
Here $H(\C)$ is equipped with its usual topology of uniform convergence on any compact subset.
\vskip 2mm
\noindent
The following result is an immediate consequence of the definition of the topology in the spaces $A_p(\C)$ for $p>0$.

\begin{proposition}\label{sect2-prop1}
Let $\bff = \{f_N\}_{N\geq 1}$ be a sequence of elements in $A_p(\C)$.
The two following assertions are equivalent:
\begin{itemize}
\item the sequence $\bff$ converges towards $0$ in
$A_p(\C)$~;
\item the sequence $\bff$ converges towards $0$ in
$H(\C)$ and there exists $A_{\bff}\geq 0$ and $B_\bff\geq 0$ such that
\begin{equation}\label{sect2-eq1}
\forall\, N\in \N^*,\quad \forall\, z\in \C,\
|f_N(z)| \leq A_{\bff}\, e^{B_{\bff} |z|^p
}.
\end{equation}
\end{itemize}
\end{proposition}

\begin{proof}
The first assertion means that there exists $B>0$ with $\lim_{N\rightarrow \infty} \|f_N\|_{B} =0$ (in particular $\|f_N\|_B \leq 1$
for $N\geq N_1$), which implies that
the sequence $\bff$ converges to $0$ in $H(\C)$ and that $|f_N(z)|
\leq A e^{B |z|^p}$ with $B$ and $A=\sup(\tilde A_1,...,\tilde A_{N_1},1)$
independent of $N$ ($\tilde A_j=\sup_\C (|f_j(z)|\, e^{-B|z|^p})$ for $j=1,...,N_1$). Conversely, assume that the second assertion holds and take $B>B_\bff$, so that, given $\varepsilon>0$, there exists
$R_\varepsilon >0$ such that
$$
\forall\, N\in \N^*,\
\sup\limits_{|z|\geq R_\varepsilon} |f_N(z)|\, e^{-B |z|^p}
\leq A_{\bff}\, e^{(B_\bff -B) R_\varepsilon^p} < \varepsilon.
$$
On the other hand, since $\bff$ converges to $0$ uniformly on any compact subset of $\C$, in particular on $\overline{D(0,R_\varepsilon)}$, there exists $N_\varepsilon \in \N^*$ such that
$$
N\geq N_\varepsilon  \Longrightarrow
\sup\limits_{|z|\leq R_\varepsilon}
|f_N(z)|\, e^{-B |z|^p} \leq \sup\limits_{|z|\leq R_\varepsilon}
|f_N(z)| < \varepsilon.
$$
Therefore $\sup_{N\geq N_\varepsilon} \|f_N\|_B <\varepsilon$ and the sequence $\bff$ converges to $0$ in $A_p(\C)$.
\end{proof}
\vskip 2mm
\noindent
To prove our main results we need an important lemma that characterizes entire functions in $A_p(\C)$ in terms of the behaviour of their Taylor development, see Lemma 2.2 in \cite{AOKI}.

\begin{lemma}\label{sect2-lem1}
The entire function $f: z\mapsto \sum_{j=0}^\infty f_j z^j$
belongs to $A_p(\C)$ if and only if there exists $C =C_f >0$ and $b= b_f >0$ such that $f\in A^p_{C,b}(\C)$, where 
\begin{equation}\label{sect2-eq2}
A^p_{C,b}(\C) = 
\Big\{\sum_{j=0}^\infty  f_j z^j \in A_p(\C)\,; \forall\, j\in \N,\  |f_j| \leq C\, \frac{b^j}{\Gamma (j/p +1)}\Big\}.
\end{equation}
\end{lemma}

\noindent
The following lemmas are refinements of results previously stated in \cite {AOKI}, except that we need here some extra dependency with respect to auxiliary parameters. They will be of crucial importance in order to prove the main results in the next sections.

\begin{lemma}\label{sect2-lem2} Let $\mathscr T$ be a set of parameters and $\tau \in \mathscr T \mapsto \D(\tau)$ be a differential operator-valued map
$$
\tau \in \mathscr T \longmapsto \D(\tau) = \sum\limits_{j=0}^\infty b_j(\tau) \Big(\frac{d}{dW}\Big)^j
$$
(with $b_j~: \mathscr T \rightarrow \C$ for $j\in \N$)
whose formal symbol
$$
\mathbb F~: (\tau,W) \in \mathscr T \times \C \longmapsto
\sum\limits_{j=0}^\infty b_j(\tau) W^j
$$
realizes for each $\tau\in \mathscr T$ an entire function of $W$ such that
\begin{equation}\label{sect2-eq2bis}
\sup\limits_{\tau \in \mathscr T,W\in \C}
\big(|\mathbb F(\tau,W)|\, e^{-B\, |W|^p}\big) = A < +\infty
\end{equation}
for some $p\geq 1$ and $B\geq 0$.
Then $\D(\tau)$ acts as a continuous operator from $A_1(\C)$ into itself uniformly with respect to the parameter $\tau \in \mathscr T$.
\end{lemma}

\begin{proof}
It follows from Lemma \ref{sect2-lem1} that the coefficient functions
$\tau \longmapsto b_j(\tau)$ satisfy then uniform estimates
$$
\forall\, j\in \N,\ \forall\, \tau \in \mathscr T,\
|b_j(\tau)| \leq C \frac{b^j}{\Gamma(j/p+1)}
$$
for some positive constants $C=C(\D)$ and $b=b(\D)$ depending only on
the finite quantity $A$ in \eqref{sect2-eq2bis} and $B$.
Let $f: W \mapsto \sum_{\ell =0}^\infty a_\ell W^\ell \in A_1(\C)$. There are then (see again Lemma \ref{sect2-lem1}) positive constants $\gamma$ and $\beta$ such that
$|a_\ell|\leq (\gamma /\ell!)\, \beta^\ell$ for any $\ell\in \N$. Consider the action of $\D$ on such $f$. One has (for the moment formally)
\begin{multline}\label{sect2-eq3}
\forall\, \tau \in \mathscr T,\
\D(\tau)(f) = \sum\limits_{j=0}^\infty
b_j(\tau) (d/dW)^j(f)= \sum\limits_{j=0}^\infty
b_j(\tau) \Big(\sum\limits_{\ell = 0}^\infty \frac{(j+\ell)!}{\ell!}\, a_{\ell+j}\, W^\ell\Big) \\
= \sum\limits_{\ell =0}^\infty \Big( \sum\limits_{j=0}^\infty
\frac{(j+\ell)!}{\ell!} b_j(\tau) a_{\ell +j}\Big)\, W^\ell
\end{multline}
with
\begin{equation}\label{sect2-eq4}
\sum\limits_{j=0}^\infty
\frac{(j+\ell)!}{\ell!} |b_j(\tau)| \, |a_{\ell +j}|
\leq \gamma\, C \frac{\beta^\ell}{\ell!}\, \sum\limits_{j=0}^\infty \frac{(b\, \beta)^j}{\Gamma(j/p+1)} = K(b, C, \beta,\gamma)\,
\frac{\beta^\ell}{\ell!}.
\end{equation}
Therefore the formal identity \eqref{sect2-eq3} is in fact a true one for any $W\in \C$, which shows that $\D(\tau)[f]\in A_1(\C)$ for any $\tau \in \mathscr T$, with
$$
\forall\, \tau \in \mathscr T,\ \forall\, W\in \C,\quad
|\D(\tau)(f)| \leq K(b,C,\beta,\gamma)\, e^{\beta |W|}.
$$
Let $\EuFrak f = \{f_N\}_{N\geq 1}$ be a sequence converging towards $0$ in $A_1(\C)$ which is equivalent to say that $\sup (b_{f_N} + C_{f_N})<+\infty$ and that $\EuFrak f$ converges towards $0$ in $H(\C)$, see Proposition \ref{sect2-prop1}. Then the sequence $\{\D(\tau)(f_N)\}_{N\geq 1}= \D(\tau)(\EuFrak f)$ is such that
$$
\forall\, N\in \N^*,\  \forall\, \tau \in \mathscr T,\ \forall\, W\in \C,\quad |\D(\tau)(f_N)(W)|\leq A_{\EuFrak f}\, e^{B_{\EuFrak f} |W|}
$$
for some positive constants $A_{\EuFrak f}$ and $B_{\EuFrak f}$ depending only on $\D$ and $\EuFrak f$. Let $\EuFrak B > B_{\EuFrak f}$ and $\varepsilon >0$. Let $R=R_\varepsilon$ large enough such that
$$
\forall\, \tau \in \mathscr T,\ \forall\, N\in \N^*,\ \forall\, W \in
\C\ {\rm with}\ |W|> R,\ |\D(\tau)(f_N)(W)| e^{- \EuFrak B |W|} \leq \varepsilon.
$$
Since $\D(\tau)(f_N)(W) = \sum_{\ell =0}^\infty a_{N,\ell}(\tau) W^\ell$
with $|a_{N,\ell}(\tau)| \leq (C_{\EuFrak f}/\ell!)\, b_{\EuFrak f}^\ell$ for some constants
$C_{\EuFrak f}$ and $b_{\EuFrak f}$ independent on $\tau\in \mathscr T$ and on $N$ (see \eqref{sect2-eq4}) and the sequence $\EuFrak f$ converges to $0$ in $H(\C)$, one can find $N=N_\varepsilon$ such that
$$
\forall\, N\geq N_\varepsilon,
\ \forall\, \tau \in \mathscr T,\ \forall\,
W\in \C\  {\rm with}\ |W|\leq R,\ |\D(\tau)(f_N)(W)| \leq \varepsilon.
$$
Hence the sequence $\D(\tau)(\EuFrak f)$ converges towards $0$ in
$A_1(\C)$, uniformly with respect to the parameter $\tau$.
\end{proof}

\vskip 2mm
\noindent
Since the next lemma involves as set of parameters $\mathscr T$ the set which is now given as split in the form $\mathscr T = \EuFrak T \times \C_Z$, where $\C_Z$ is already a copy of the complex plane, one needs to duplicate $\C_Z$ into an extra copy of $\C$ denoted as $\C_W$.

\begin{lemma}\label{sect2-lem3} Let $\EuFrak T$ be a set of parameters and $\EuFrak t \in \EuFrak T\mapsto \D(\EuFrak t,Z)$ be a differential operator-valued map
$$
\EuFrak t \in \EuFrak T \longmapsto \D(\EuFrak t,Z) = \sum\limits_{j=0}^\infty b_j(\EuFrak t,Z) \Big(\frac{d}{dZ}\Big)^j
$$
(with $b_j~: \EuFrak T\times \C \rightarrow \C$, holomorphic in $Z$ for $j\in \N$) such that
\begin{equation}\label{sect2-eq4bis}
\forall\, \varepsilon >0,\quad
\sup\limits_{\EuFrak t\in \EuFrak T, (Z,W)\in \C^2}
\Big(\Big(\sum\limits_{j=0}^\infty |b_j(\EuFrak t,Z)|\, |W|^j\Big)\,
\exp (- \varepsilon\, |Z|^{\check p} - B\, |W|^{p})\Big) = A^{(\varepsilon)} < +\infty
\end{equation}
for some $\check p>1$, $p \geq 1$ and $B\geq 0$.
Then $\D(\EuFrak t,Z)$ acts as a continuous operator from $A_1(\C)$ into $A_{\check p,0}(\C)$ uniformly with respect to the parameter $\EuFrak t\in \EuFrak T$.
\end{lemma}

\begin{proof}
The function
\begin{equation}\label{sect2-eq5}
\F~: (\EuFrak t,Z,W) \longmapsto \sum\limits_{j=0}^\infty
\Big(\sum\limits_{\kappa =0}^\infty b_{j,\kappa}(\EuFrak t) Z^\kappa\Big)\, W^j =
\sum\limits_{\kappa =0}^\infty Z^\kappa\, \Big(
\sum\limits_{j=0}^\infty b_{j,\kappa}(\EuFrak t) \, W^j\Big)
\end{equation}
is well defined and depends as an entire function of two variables of the variables $Z$ and $W$ (which also justifies in \eqref{sect2-eq5} the application of Fubini theorem).
Cauchy formulae in $\C \times \C$ show that for any $\EuFrak t\in \EuFrak T$, for any $j,\kappa\in \N$,
\begin{multline}\label{sect2-eq6}
|b_{j,\kappa}(\EuFrak t)| = \frac{1}{4\pi^2}\Big|
\int_{|Z|=\check r,|W|=r} F(\EuFrak t,Z,W)
\frac{dZ}{Z^{\kappa+1}} \wedge \frac{dW}{W^{j+1}}\Big|
\leq A^{(\varepsilon)}\, \inf_{\check r>0}\frac{e^{\varepsilon \check r^{\check p}}}{r^\kappa}\times \inf_{r>0} \frac{e^{B r^p}}{r^j} \\
= A^{(\varepsilon)}
\Big(\frac{1}{\kappa}\Big)^{\kappa/\check p}
 \times \Big(\frac{1}{j}\Big)^{j/p}\, ((\varepsilon \check p\, e)^{1/\check p})^\kappa\, ((Bpe)^{1/p})^j \\
\leq C_{\eta}\,
 \frac{1}{\Gamma(\kappa/\check p+1) \Gamma(j/p+1)}\, (\eta\, \check b)^\kappa\, b^j
\end{multline}
for each $\eta>0$, with constants $C_{\eta}$, $\check b$ and $b$ independent on the parameter $\EuFrak t$. Let now $\EuFrak f = \{f_N\}_{N\geq 1}$ be a sequence of elements in
$A_1(\C)$ which converges to $0$ in $A_1(\C)$.
All differential operators
$$
\D_{\kappa}(\EuFrak t)~: = \sum\limits_{j=0}^\infty
b_{j,\kappa}(\EuFrak t) (d/dW)^j\quad (\kappa \in \N)
$$
act continuously on $A_1(\C)$, as seen in Lemma \ref{sect2-lem2}.
Moreover, one has (plugging in \eqref{sect2-eq4}
the estimates \eqref{sect2-eq6}) that
\begin{multline*}
\forall\, f\in A^1_{\gamma,\beta}(\C),\ \forall\, \EuFrak t\in \EuFrak T,\
\forall\, \kappa \in \N,\
\forall\, \ell\in \N,\\
(\D_{\kappa}(\EuFrak t)(f))_\ell
\leq \gamma\, \tilde C_{\eta}
\frac{(\eta\, \check b)^\kappa}{\Gamma(\kappa/\check p+1)}\, E_{1/p,1}
(\beta\, b)\, \frac{\beta^\ell}{\ell!}
\end{multline*}
where $E_{1/p,1}~: \zeta \in \C \longmapsto \sum_0^\infty \zeta^k/\Gamma (k/p + 1)$ is the entire (with order $1/p$ and type $1$) Mittag-Leffler function. One has therefore for such $f\in A^1_{\gamma,\beta}(\C_W)$ that
\begin{equation}\label{sect2-eq7}
\forall\, \EuFrak t \in \EuFrak T,\ \forall\, \kappa \in \N,\
\forall\, W\in \C,\quad
|\D_{\kappa}(\EuFrak t)(f)(W)|
\leq \gamma\, C_{\eta}\, E_{1/p,1}(\beta\, b)\, \frac{e^{\beta |W|}}{\Gamma (\kappa/\check p +1)}
\end{equation}
and (taking now $W=Z$)
\begin{equation}\label{sect2-eq8}
\forall\, \EuFrak t \in \EuFrak T,\
\forall\, Z \in \C,\quad
\sum\limits_{\kappa =0}^\infty
|Z|^\kappa\, |\D_{\kappa}(\EuFrak t)(f)(Z)|
\leq
\gamma\, C_{\eta}\, E_{1/p,1}(\beta\, b)\, e^{\beta |Z|}\,
\sum\limits_{\kappa =0}^\infty \frac{(\eta\,\check b\, |Z|)^\kappa}{\Gamma (k/\check p +1)}.
\end{equation}
Since the Mittag-Leffler function $E_{1/\check p,1}$ has order
$\check p>1$, the estimates \eqref{sect2-eq8} (uniform in
the parameter $\EuFrak t$ as well as on the function $f\in A^1_{\gamma,\beta}(\C)$) show that the differential operator acts continously
from $A_1(\C)$ into $A_{\check p,0}(\C)$, uniformly with respect to the parameter
$\EuFrak t\in \EuFrak T$. One just needs to repeat here the end of the proof of Lemma \ref{sect2-lem2}.
\end{proof}
\vskip 2mm
\noindent
We conclude this section by proving
a quantitative lemma which reveals to be essential in the sequel.
It is a refinement of  Lemma 1 in \cite{CSSY}.

\begin{lemma}\label{sect2-lem5} Let $a\in \C$ with $\alpha := \max(1,|a|)$ and, for any $z\in \C$,
$$
F_N(z,a) :=
\Big(\cos \Big(\frac{z}{N}\Big) + i\, a\, \sin\Big(\frac{z}{N}\Big)\Big)^N
$$
as in \eqref{sect1-eq1} (with $z, a\in \C$ instead of $x, a\in \R$).
For any $N\in \N^*$ and any $z\in \C$, one has
\begin{equation}\label{sect2-eq9}
\begin{split}
& |F_N(z,a)| \leq \exp\big(|a|\, |z| + |{\rm Im}(z)|\big)\leq
\exp\big((|a|+1)\, |z|\big)
 \\
& |F_N(z,a) - e^{iaz}| \leq
\frac{2}{3}\,
\frac{|a^2-1|}{N}\, |z|^2\, \exp\big( (\alpha+1) |z|\big).
\end{split}
\end{equation}
\end{lemma}

\begin{proof}
Let
$$
{\rm sinc}: z \in \C \longmapsto \frac{\sin z}{z} =
\int_0^1 t\, \cos (tz)\, dt
$$
be the sinus cardinal function; it satisfies $|{\rm sinc}(z)|
\leq e^{|{\rm Im}(z)|}$ for any $z\in \C$. One has then the upper uniform estimates
\begin{multline}\label{sect2-eq10}
\forall\, N\in \N^*,\ \forall\, z\in \C,\ |F_N(z,a)| =
\Big|
\cos \Big( \frac{z}{N}\Big) + i a
\sin\Big(\frac{z}{N}\Big)\Big|^N =
\Big|
\cos \Big(\frac{z}{N}\Big) + i\, \frac{az}{N}\, {\rm sinc}
\Big(\frac{z}{N}\Big)\Big|^N \\
\leq e^{|{\rm Im}(z)|} \Big( 1 + \frac{|az|}{N}\Big)^N
\leq \exp (|a|\, |z| + |{\rm Im}(z)|) \leq \exp \big((|a|+1)|z|\big),
\end{multline}
which is the first chain of inequalities in \eqref{sect2-eq9}.
For any $N\in \N^*$, one has also
\begin{multline}\label{sect2-eq11}
\Big|\cos \Big(\frac{z}{N}\Big) - \cos \Big(\frac{a z}{N}\Big)\Big|
=
2\,\Big|\sin \Big(
\frac{(a-1)z}{2N}\Big)\, \sin \Big(\frac{(a+1)z}{2N}\Big)\Big| \\
\leq \frac{|a^2-1|}{2N^2}\, |z|^2\, \exp
\Big(\frac{|a-1|+|a+1|}{2N}\, |z|\Big)
\leq \frac{|a^2-1|}{2N^2}\, |z|^2\, \exp \Big(\frac{\alpha+1}{N}\, |z|\Big)
\end{multline}
and
\begin{multline}\label{sect2-eq12}
\Big|
a \sin \Big(
\frac{z}{N}\Big) -
\sin \Big(
\frac{a z}{N}\Big)\Big|
=
\Big|
\sum\limits_{k=0}^\infty
\frac{(-1)^{k}}{(2k+1)!} (a-a^{2k+1})\, \Big(
\frac{z}{N}\Big)^{2k+1}\Big| \\
= \frac{|a^2-1|}{N^2}\, |z|^2\,
\Big|
\sum\limits_{k=1}^\infty
\frac{(-1)^{k}}{(2k+1)!}\, \Big(\sum\limits_{\ell=0}^{k-1}
a^{2\ell+1}\Big)
\, \Big(
\frac{z}{N}\Big)^{2k-1}\Big| \\
\leq \frac{|a^2-1|}{2N^2}\, |z|^2
\sum\limits_{k=1}^\infty
\frac{\alpha^{2k-1}}{(2k-1)!(2k+1)}\, \Big(
\frac{|z|}{N}\Big)^{2k-1} \\
\leq \frac{|a^2-1|}{6N^2}\, |z|^2
\, \sum\limits_{k=1}^\infty
\frac{1}{(2k-1)!}\, \Big(
\frac{\alpha |z|}{N}\Big)^{2k-1} \leq \frac{|a^2-1|}{6N^2}\, |z|^2
\, \exp \Big(
\frac{\alpha}{N}\, |z|\Big).
\end{multline}
It follows from the identity $A^N-B^N = (A-B) \sum\limits_{k=0}^{N-1}
A^k B^{N-1-k}$, together with estimates
\eqref{sect2-eq11}, \eqref{sect2-eq12} and \eqref{sect2-eq10},
that for any $N\in \N^*$ and $z\in \C$,
\begin{multline*}
|F_N(z,a) - e^{iaz}| =
\Big|\cos \Big(\frac{z}{N}\Big) - \cos \Big(\frac{a z}{N}\Big)
+ i \Big(a \sin \Big(
\frac{z}{N}\Big) -
\sin \Big(
\frac{a z}{N}\Big)\Big)\Big| \\
\times \sum\limits_{k=0}^{N-1}
|F_N(z,a)|^{k/N} \,
\Big|\exp\big( ia z \frac{N-1-k}{N}\big)\Big| \\
\leq \frac{2}{3}\, \frac{|a^2-1|}{N^2}\, |z|^2\,
\exp\Big(\frac{\alpha+1}{N}\, |z|\Big)\,
\sum\limits_{k=0}^{N-1}
\exp\Big(k\, \Big(\frac{|a|+1}{N}\Big)\, |z|
+ \frac{N-1-k}{N}\, |a|\, |z|
\Big)\\
\leq \frac{2}{3}\, \frac{|a^2-1|}{N}\, |z|^2\, \exp\big(
(\alpha+1)|z|\big).
\end{multline*}
The second inequality in \eqref{sect2-eq9} is thus proved.
\end{proof}
\vskip 2mm
\noindent
One can now state as a consequence of Proposition \ref{sect2-prop1} and Lemma \ref{sect2-lem5} the following theorem.

\begin{theorem}\label{sect2-thm1} For any $a\in \C$,
the sequence $\{z\mapsto F_N(z,a)\}_{N\geq 1}$ converges to
$z\mapsto e^{iaz}$ in $A_1(\C)$.
\end{theorem}

\begin{proof} It follows from estimates \eqref{sect2-eq10} that the sequence $\bff = \{z\mapsto F_N(z,a)\}_{N\geq 1}$ satisfies the estimates
\eqref{sect2-eq1} with $p=1$, $B_\bff = |a|+1$ and $C_\bff = 1$.
Lemma \ref{sect2-lem5} implies on the other hand that the sequence $\bff$ converges
towards $z\mapsto e^{iaz}$ in $H(\C)$. The result is then a consequence of Proposition \ref{sect2-prop1}.
\end{proof}

\section{Uniform convergence of superoscillating sequences}
\label{sect3}

Let $m\in \N^*$ and $(\mathscr F(\R^m,\C))^{\N^*}$ be the
family of all sequences $\boldsymbol Y=\{x\in \R^m \mapsto Y_N(x)\}_{N\geq 1}$ of complex valued functions defined on $\R^m$. We first recall in this section the notions of (complex) {\it generalized Fourier sequence}
(CGFS) and (complex) {\it superoscillating sequence} (CSOscS) in
$(\mathscr F(\R^m,\C))^{\N^*}$. We start first with the case $m=1$.

\begin{definition}\label{sect3-def1}
A sequence $\boldsymbol Y \in (\mathscr F(\R,\C))^{\N^*}$ is called
a {\rm complex generalized Fourier sequence} if each entry $Y_N$ is, after re-indexation, of the form
\begin{equation}\label{sect3-eq1}
Y_N : x\in \R
\longmapsto \sum\limits_{j=0}^N  C_{j}(N)\, \exp (i k_{j}(N)x),
\end{equation}
where $C_j(N)\in \C$ and $k_j(N)\in \R$ for any $N\in \N^*$ and $j\in \N$.
\end{definition}

\begin{example}\label{sect3-expl1}{\rm
\vskip 1mm
\noindent
\begin{enumerate}
\item If $f \in L^1\big(\T,\C\big)$, where $\T =\R/(2\pi\Z)$, is any subsequence  of the Fourier (resp. Fourier-Fej\'er) sequences
$\{x\mapsto S_N(x)\}_{N\geq 1}$ (resp. $\{x\mapsto F_N(x)\}_{N\geq 1}$), where
\begin{eqnarray*}
S_N(x) &=& \sum\limits_{j=0}^{2N}
\Big(
\int_{\T} f(\theta)
e^{-i\, (j-N)\, \theta}\,
\frac{d\theta}{(2\pi)}\Big)\, e^{i\, (j-N)\, x}
\\
F_N(x)
&=&
\sum\limits_{j=0}^{2N}
\Big(1 - \frac{|j-N|}{N}\Big)\, \Big(
\int_{\T} f(\theta)
e^{-i\, (j-N)\, \theta}\,
\frac{d\theta}{(2\pi)}\Big)\, e^{i\, (j-N)\, x},
\end{eqnarray*}
then it
realizes, after re-indexation, an archetypical example of a complex generalized Fourier sequence
in $(\mathscr F(\R,\C))^{\N^*}$ This fact justifies the terminology.
\item
When $m=1$ and $a\in \R$, the sequence $\{x\mapsto F_N(x,a)\}_{N\geq 1}$
is also an example of a complex generalized Fourier sequence in $(\mathscr F(\R,\C))^{\N^*}$. In this case, note that $C_j(N) = C_j(N,a)\in \R$ for any $j\in \N$.
\item Let $P= \sum_{\kappa \in \Z^*} \gamma_\kappa X^\kappa \in \C[X,X^{-1}]$ be a Laurent polynomial and $L(P)$ the diameter of its support. Any sequence $\{x\mapsto Y_N(x)\}_{N\geq 1}$ such that
$$
Y_N(x)= \sum_{j=0}^N C_j(N) P(e^{ik_j(N) x}) =
\sum\limits_{j=0}^N
\sum\limits_{\kappa\in \Z^*}\lambda_\kappa C_j(N)
e^{i\kappa \kappa_j(N) x} = \sum\limits_{j=0}^{ L(P)\, N}
\widetilde C_j(N)\, e^{i\widetilde \kappa_j(N) x}
$$
is after re-indexation a complex generalized Fourier sequence in $(\mathscr F(\R,\C))^{\N^*}$.
\end{enumerate}
}
\end{example}

\begin{definition}\label{sect3-def2}
A complex generalized Fourier sequence $\{x\mapsto Y_N(x)\}_{N\geq 1}$ in $(\mathscr F(\R,\C))^{\N^*}$
is called a {\rm complex superoscillating sequence} if
\begin{itemize}
\item each entry $Y_N$ is of the form \eqref{sect3-eq1} with $|k_j(N)|\leq 1$ for any $j\in \N$ such that $0\leq j\leq N$;
\item
there exists an open subset $U^{\rm sosc}\subseteq \R$ which is called a
{\rm superoscillation domain} such that $\{x\mapsto Y_N(x)\}_{N\geq 1}$
converges uniformly on any compact subset of $U^{\rm sosc}$ to the restriction to $U^{\rm sosc}$ of a trigonometric polynomial function
$$
Y_\infty: x \longmapsto  P_\infty(e^{ik(\infty) x})
$$
where $P_\infty\in \C[X,X^{-1}]$ is a Laurent polynomial with no constant term and $k(\infty) \in \R\setminus [-1,1]$.
\end{itemize}
\end{definition}

\begin{remark}\label{sect3-rem1}
{\rm
If $\boldsymbol Y$ is a superoscillating sequence in the sense of Definition \ref{sect3-def2}, it is $Y_\infty$-superoscillating in the sense of Definition 1.1 in \cite{CSY}, with {\it superoscillation set} any segment $[a,b]$ such that $b-a>0$ is included in the superoscillation domain $U^{\rm sosc}$.
}
\end{remark}

\begin{example}\label{sect3-expl2}
{\rm
\vskip 1mm
\noindent
\begin{enumerate}
\item Any subsequence of the Fourier (resp. Fourier-Fej\'er) sequences
$\{x\mapsto S_N(x)\}_{N\geq 1}$ (resp. $\{x\mapsto F_N(x)\}_{N\geq 1}$) introduced in Example \ref{sect3-expl1} (1) fails to be superoscillating since the condition $|k_j(N)|\leq 1$ is not fulfilled.
\item
If $a\in \R \setminus [-1,1]$, the sequence
$\{x\mapsto F_N(x,a)\}_{N\geq 1}$ is a superoscillating sequence
in $(\mathscr F(\R,\C))^{\N^*}$ with superoscillation domain equal to $\R$, with $Y_\infty: x\in \R \mapsto e^{iax}$. This follows from
Lemma \ref{sect2-lem5} (namely from the inequalities \eqref{sect2-eq9} for $a\in \R$ and $x\in \R$). This is the model that inspired us originally and that we will generalize in this paper.
\end{enumerate}
}
\end{example}

\vskip 2mm
\noindent
Inspired by physical considerations which we will discuss later on,
we extend as follows Definition \ref{sect3-def1} and Definition
\ref{sect3-def2} to the higher dimensional setting where $m>1$.
The model we will use in order to extend Definition \ref{sect3-def1} will be the one in Example \ref{sect3-expl1} (3).

\begin{definition}\label{sect3-def3}
A sequence $\boldsymbol Y \in (\mathscr F(\R^m,\C))^{\N^*}$ is called
a complex {\rm generalized Fourier sequence} if, after re-indexation, each entry $Y_N$ is of the form
\begin{equation}\label{sect3-eq2}
Y_N : x=(x_1,...,x_m)\in \R^m
\longmapsto \sum\limits_{j=0}^N  C_{j}(N)\,
P\big(e^{ix_1 k_{j,1}(N)},...,e^{ix_m k_{j,m}(N)}\big)
\end{equation}
where $P\in \C[X_1,...,X_m]\in \C[X_1^{\pm 1},...,X_m^{\pm 1}]$ is a Laurent polynomial (independent of $N$), $C_j(N)\in \C$ and $N\mapsto k_j(N)$ is a map from $\N^*$ to
$\R^m$ for any $N\in \N^*$ and $j=0,...,N$.
\end{definition}

\begin{example}\label{sect3-expl3}
{\rm Let $t,x$ be two real variables, $C_j(N)\in \C$,
$\kappa_j(N)\in \R$, $k_j(N)\in \R$ for any $N\in \N^*$ and
$0\leq j\leq N$. Then
$$
\Big\{x \mapsto \sum\limits_{j=0}^N C_j(N) e^{i \kappa_j(N)\, t} e^{i k_j(N)\, x}\Big\}_{N\geq 1}
$$
is a complex generalized Fourier sequence in the two real variables $t,x$, the polynomial $P\in \C[T,X]$ being here $P(T,X)=TX$.
}
\end{example}

\noindent
Definition \ref{sect3-def2} extends to the multivariate case as follows.

\begin{definition}\label{sect3-def4}
A complex generalized Fourier sequence $\{x\mapsto Y_N(x)\}_{N\geq 1}$ in $(\mathscr F(\R^m,\C))^{\N^*}$
is called a {\rm complex superoscillating sequence} if
\begin{itemize}
\item each entry $Y_N$ is of the form \eqref{sect3-eq2} with additionally $|k_{j,\ell}(N)|\leq 1$ for any $j\in \N$ such that $0\leq j\leq N$ and $\ell=1,...,m$;
\item
there exists an open subset $U^{\rm sosc}\subseteq \R^m$ which is called a
{\rm superoscillation domain} such that $\{x\mapsto Y_N(x)\}_{N\geq 1}$
converges uniformly on any compact subset of $U^{\rm sosc}$ to the restriction to $U^{\rm sosc}$ of a trigonometric polynomial function
$$
Y_\infty: x \longmapsto  P_\infty(e^{ik_1(\infty) x_1},...,e^{i k_m(\infty) x_m})
$$
where $P_\infty\in \C[X_1^{\pm 1},...,X_m^{\pm 1}]$ is a Laurent polynomial with no constant term and $k_j(\infty) \in (\R\setminus [-1,1])^m$.
\end{itemize}
\end{definition}
\noindent
In order to illustrate Definition \ref{sect3-def4} with an example which is derived from Example \ref{sect3-expl2} (2), consider, for $p\in \N$ and $a\in \R \setminus [-1,1]$, the complex generalized Fourier sequence in two real variables $t,x$
\begin{equation}\label{sect3-eq3}
\Big\{\psi_{p,N}(\cdot,\cdot,a)~: (t,x) \in \R^2 \longmapsto
\sum\limits_{j=0}^N C_j(N,a)\, e^{i (1-2j/N)^p\, t}\, e^{i(1-2j/N)x}\Big\}_{N\geq 1}
\end{equation}
(see Example \ref{sect3-expl3}). An immediate computation shows that for any $(t,x)\in \R^2$,
\begin{eqnarray*}
\frac{\partial}{\partial t}
\big(\psi_{p,N}(t,x,a)\big) &=& i\, \sum\limits_{j=0}^N
C_j(N,a)\, (1-2j/N)^p\,e^{i (1-2j/N)^p\, t}\, e^{i(1-2j/N) x}  \\
\frac{\partial^p}{\partial x^p}
\big(\psi_{p,N}(t,x,a)\big) &=& i^p\, \sum\limits_{j=0}^N
C_j(N,a)\, (1-2j/N)^p\,e^{i (1-2j/N)^p\, t}\, e^{i(1-2j/N) x},
\end{eqnarray*}
which shows that $(t,x)\in \R^2 \longmapsto \psi_{p,N}(t,x,a)$ is the (unique) global solution of the Cauchy-Kowalevski problem
\begin{equation}\label{sect3-eq4}
\Big(i^{p-1} \frac{\partial}{\partial t} -
\frac{\partial^p}{\partial x^p}\Big) (\psi) \equiv 0,\quad
[\psi(t,x)]_{|t=0} = F_N(x,a).
\end{equation}
One can extend analytically $\psi_{p,N}(\cdot,\cdot,a)$ as a function
from $\R \times \C$ to $\C$, such that one has formally
\begin{eqnarray}\label{sect3-eq5}
\psi_{p,N}(t,z,a) &=& \sum\limits_{j=0}^N C_j(N,a)\,
\Big( \sum\limits_{\ell=0}^\infty
\frac{i^{\ell(1-p)}\, t^\ell}{\ell!} \big(i(1-2j/N)\big)^{p\ell}\Big)\, e^{i(1-2j/N)z} \nonumber \\
&=& \Big(\sum\limits_{\ell=0}^\infty
\frac{i^{\ell(1-p)}\, t^\ell}{\ell!} D^{p\ell}\Big)(F_N(\cdot,a)) (z) = \D_p(t) (F_N(\cdot,a)) (z).
\end{eqnarray}
One can prove here the following result.

\begin{theorem}\label{sect3-thm1}
The operator $\D_p(t)$ acts continuously from $A_1(\C)$ into itself.
The generalized Fourier sequence \eqref{sect3-eq3} is superoscillating
with $\R^2$ as superoscillation domain and limit function
$$
Y_\infty~: (t,x) \longmapsto e^{it a^p}\, e^{iax},
$$
($P_\infty(T,X) = T X$, $k_1(\infty) = a^p$, $k_2(\infty)=a$) uniformly on any compact in $\R^2$. For any $(\mu,\nu)\in \N^2$, the sequence of functions
\begin{multline}\label{sect3-eq6}
\frac{\partial^{\mu +\nu}}{\partial t^\mu
\partial x^\nu} (\psi_{p,N}(t,x,a)) =
i^{-\mu(1-p)} \frac{\partial^{p\mu + \nu}}{\partial x^{p\mu+\nu}}
(\psi_{p,N}(t,x,a)) \\
= i^{-\mu(1-p)}
((d/dW)^{p\mu+\nu}\odot \D_p(t))(F_N(\cdot,a)(x) \quad (N\in \N^*)
\end{multline}
converges uniformly on any compact in $\R^2$ to the function
$$
(t,x)\in \R^2 \mapsto ((d/dW)^{p\mu+\nu}\odot \D_p(t))(e^{ia^p t}\, e^{ia (\cdot)})(x).
$$
\end{theorem}

\begin{proof} The first assertion follows from Lemma \ref{sect2-lem2}
with $\R_t= \mathscr T$ as set of parameters and $p\geq 1$ as order of the symbol of the differential operator $\D_p(t)$ as a differential operator in $W$. Since $\{z\mapsto F_N(z,a)\}_{N\geq 1}$ converges to $z\mapsto e^{iaz}$ in $A_1(\C)$ (see theorem \ref{sect2-thm1}), the sequence $\{z\mapsto \D_p(t)(F_N(\cdot,a))(z)\}_{N\geq 1}$ converges towards $z\mapsto \D_p(t)(e^{ia(\cdot)})(z)$ locally uniformly with respect to $t\in \R$. One can check that
$\D_p(t)(e^{ia(\cdot)})(z) = e^{ia^pt} e^{iaz}$ thanks to an immediate computation. Since
$(1-2j/N)^p$ and $(1-2j/N)$ lie in $[-1,1]$ for any $j\in \{0,...,N\}$
and $a\in \R \setminus [-1,1]$, the generalized Fourier sequence \eqref{sect3-eq3} is superoscillating with $P_\infty(T,X)=TX$,
$k_1(\infty)=a^p$ and $k_2(\infty)=a$, the superoscillation domain being here $\R^2$. The expressions of the partial derivatives in $t$ in terms of the partial derivatives in $x$ in \eqref{sect3-eq6} follow from the fact that $\psi_{p,N}(\cdot,\cdot,a)$ satisfies the partial differential equation in the Cauchy-Kowalevski problem \eqref{sect3-eq4}. The last assertion in the theorem results from the continuity of
the differentiation $d/dz$ as an operator from $A_1(\C)$ into itself.
\end{proof}
\vskip 2mm
\noindent
Let now $P\in \R[X]$ be an even polynomial $P(X)=\gamma_0 + \gamma_1 X^2 +\cdots + \gamma_{2d'} X^{2d'}$ and $a\in \R\setminus [-1,1]$. Consider in this case the generalized Fourier sequence
\begin{equation}\label{sect3-eq7}
\Big\{\psi_{P,N}(\cdot,\cdot,a)~: (t,x) \in \R^2 \longmapsto
\sum\limits_{j=0}^N C_j(N,a)\, e^{i P(1-2j/N)\, t}\, e^{i(1-2j/N)x}\Big\}_{N\geq 1}.
\end{equation}
As in the previous case, an easy computation shows that the function $\psi_{P,N}(\cdot,\cdot,a)$ is the unique global solution (on the whole space $\R^2$) of the Cauchy-Kowalevski problem
\begin{equation}\label{sect3-eq8}
\Big(i \frac{\partial}{\partial t} -
\check P\Big(\frac{\partial}{\partial x}\Big)\Big) (\psi) \equiv 0,\quad
[\psi(t,x)]_{|t=0} = F_N(x,a)
\end{equation}
where $\check P=\sum_{\kappa'=0}^{d'} (-1)^{\kappa'+1} \gamma_{2\kappa'}\, X^{2\kappa'}$, and the partial differential operator is here of Schr\"odinger type.
Let us introduce the differential operator $\D_P(t)$ defined as
$$
\D_P(t) = \bigodot_{\kappa' = 0}^{d'}
\big(\sum\limits_{\ell =0}^\infty
\frac{(i^{1-2\kappa'}\, t \gamma_{2\kappa'})^\ell}{\ell!}\, (d/dW)^{2\kappa'\ell}\big)
$$
with symbol in $A_{2d'}(\C_W)$ (the set of parameters $\mathscr T$ being again $\mathscr T=\R_t$).

\begin{theorem}\label{sect3-thm2}
Let $P\in \R[X]$ be an even polynomial with degree $2d'$. 
For any $\lambda \in \R$, the Cauchy-Kowalevski problem (of Schr\"odinger type)
\begin{equation}\label{sect3-eq9}
\Big(i \frac{\partial}{\partial t} -
\check P\Big(\frac{\partial}{\partial x}\Big)\Big) (\psi) \equiv 0,\quad
[(t,x)\mapsto \psi(t,x)]_{|t=0} = [x\mapsto e^{i\lambda x}]
\end{equation}
admits as unique global solution in $\R^2$ the function
$(t,x) \mapsto \varphi_{\lambda} (t,x) = e^{it P(\lambda)} e^{i\lambda x}$. One has $\psi_{P,N}(\cdot,\cdot,\lambda) =
\sum_{j=0}^N C_j(N,\lambda)\, \varphi_{1-2j/N}$ and the sequence
$\{(t,x)\mapsto \psi_{P,N}(t,x,\lambda)\}_{N\geq 1}$ converges uniformly on any compact set in $\R^2$ to
$(t,x) \mapsto e^{itP(\lambda)} e^{i\lambda x}$.
For any $(\mu,\nu)\in \N^2$, the sequence of functions
\begin{multline}\label{sect3-eq10}
\frac{\partial^{\mu +\nu}}{\partial t^\mu
\partial x^\nu} (\psi_{P,N}(t,x,\lambda)) =
(-i)^\mu \Big(\check P\Big(
\frac{\partial}{\partial x}\Big)\Big)^{\odot^\mu}
\odot \Big(\frac{\partial}{\partial x}\Big)^{\odot^\nu} (\psi_{p,N}(t,x,\lambda)\\
= (-i)^\mu
\Big(\big(\check P (d/dW)\big)^{\odot^\mu}
\odot (d/dW)^\nu
\odot \D_P(t)\Big)\big(F_N(\cdot,\lambda)\big)(x) \quad (N\in \N^*)
\end{multline}
converges uniformly on any compact in $\R^2$ to the function
$$
(t,x)\in \R^2 \mapsto
(-i)^\mu
\Big(\big(\check P (d/dW)\big)^{\odot^\mu}
\odot \big(d/dW\big)^{\odot^\nu}
\odot \D_P(t)\Big)
\big(e^{iP(\lambda) t}\, e^{i\lambda (\cdot)}\big)(x).
$$
\end{theorem}

\begin{proof} One has
$$
\Big(i \frac{\partial}{\partial t} -
\check P\Big(\frac{\partial}{\partial x}\Big)\Big) (\varphi_{\lambda}) = (- P(\lambda) + P(\lambda))\, e^{it P(\lambda)} e^{i\lambda x} \equiv 0
$$
and $\varphi_{\lambda}(0,x)=e^{i\lambda x}$ for all $x\in \R$.
It follows from Lemma \ref{sect2-lem2} that the operator $\D_P(t)$ acts continuously on $A_1(\C)$, locally uniformly with respect to the parameter $t\in \R$. Since the sequence $\{z\in \C \mapsto F_N(\cdot,\lambda)\}_{N\geq 1}$ converges to $z\mapsto e^{i\lambda z}$ in $A_1(\C)$, the sequence $\{z\in \C \mapsto \D_P(t)(F_N(\cdot,\lambda))(z)\}_{N\geq 1}$ converges to $z\mapsto \D_P(t)(e^{i\lambda (\cdot)})(z)=e^{it P(\lambda)} e^{i\lambda z}$ in $A_1(\C)$ locally uniformly with respect to the parameter $t\in \R$. The first equality in \eqref{sect3-eq10}
follows from the fact that $\psi_{P,N}(\cdot,\cdot,\lambda)$
is solution of the Cauchy-Kowalevski problem \eqref{sect3-eq8}. The final assertion follows from the continuity of $d/dz~: A_1(\C)\rightarrow A_1(\C)$.
\end{proof}

\noindent
One can even drop the hypothesis about $P$ and take $P=\sum_{\kappa=0}^d \gamma_\kappa X^\kappa$ as polynomial of degree $d$ in $\C[X]$ with associate polynomial $\check P = \sum_{\kappa=0}^d (-i)^{\kappa+1} \gamma_\kappa X^\kappa$. The Cauchy-Kowalevski problem
\eqref{sect3-eq8} is not anymore of the Schr\"odinger type (since
$\check P \notin \R[X]$ in general), which makes the only difference with the case previously studied. Nevertheless, one can state exactly the same result, with this time
$$
\D_P(t) =  \bigodot_{\kappa = 0}^{d}
\big(\sum\limits_{\ell =0}^\infty
\frac{(i^{1-\kappa}\, t \gamma_{\kappa})^\ell}{\ell!}\, (d/dW)^{\kappa\ell}\big).
$$

\begin{theorem}\label{sect3-thm3}Let $P\in \C[X]$ be a polynomial of degree $d$.
All the assertions in Theorem \ref{sect3-thm2} are valid, except that
\eqref{sect3-eq8} is not  anymore a Cauchy-Kowalevski problem of the Schr\"odinger type. When $a\in \R\setminus [-1,1]$, the generalized Fourier sequence $\{x\mapsto \psi_{B,N}(t,x,a)\}_{N\geq 1}$ is superoscillating for any $t\in \R$. Moreover, given such $a$ and $P\in \R[X]$ such that $\sup_{[-1,1]} |P| \leq 1 < |P(a)|$, the generalized Fourier sequence
$$
\Big\{(t,x) \mapsto \psi_{P,N}(t,x,a) = \sum\limits_{j=0}^N
C_j(N,a)\, \varphi_{1-2j/N}(t,x)\Big\}_{N\geq 1}
$$
is superoscillating as a generalized Fourier sequence in two variables $(t,x)$, with $\R^2$ as domain of superoscillation.
\end{theorem}

\begin{proof}
The proof follows that one of Theorem \ref{sect3-thm2}.
The fact that the sequence $\{x\mapsto \psi_{B,N}(t,x,a)\}_{N\geq 1}$ is superoscillating for any $t\in \R$ follows from the fact that it converges on any compact of $\R_x$ (locally uniformly in $t$)
to $x\mapsto e^{itP(a)}\, e^{iax}$.
As for the last assertion, to define $Y_\infty$ one takes $P_\infty (T,X)=T X$, $\kappa(\infty) = P(a)$ and $k(\infty) =a$ in Definition
\ref{sect3-def4}.
\end{proof}

\noindent
Let now $E(X) = \sum_{\kappa=0}^\infty \gamma_\kappa X^\kappa \in \C[[X]]$ be a power series with radius of convergence $\rho\in ]0,+\infty]$, together with the convolution operator
$$
\D_E(t) := \lim\limits_{d\rightarrow +\infty}
\bigodot\limits_{\kappa=0}^d \Big(\sum\limits_{\ell =0}^\infty
\frac{(i^{1-\kappa} t \gamma_\kappa)^\ell}{\ell!}
\, (d/dW)^{\kappa \ell}\Big)
$$
with formal symbol
$$
\F_E(t)~: W \longmapsto \exp \Big( it \sum\limits_{\kappa=0}^\infty
i^{1-\kappa}\, \gamma_\kappa W^\kappa\Big).
$$
Since $F$ and $\sum_{\kappa =0}^\infty i^{1-\kappa} \gamma_\kappa X^\kappa$ share the same radius of convergence $\rho>0$,
$\F_E(t)$ realizes, for each $t\in \R$) an holomorphic function in $D(0,\rho) \subset \C_W$ (with Taylor series about $0$ depending on $t\in \R$).
More precisely, one has
$$
\forall\, t,W \in \R \times D(0,\rho),\quad
\F_E(t)(W) = \sum\limits_{j=0}^\infty \Big(\sum\limits_{\kappa=0}^\infty b_{j,\kappa} t^\kappa\Big) W^j = \sum\limits_{j=0}^\infty b_j (t)\, W^j,
$$
where, for $R>0$, the radius of convergence of the power series
$\sum_{j\geq 0} \Big(\sum_{\kappa \geq 0} |\beta_{j,\kappa}| R^\kappa\Big) X^j$ is at least equal to $\rho$.
\vskip 2mm
\noindent
For any $\lambda \in ]-\rho,\rho[$ and $z\in \C$, one has formally
\begin{multline}\label{sect3-eq11}
e^{it E(\lambda)} e^{iz\lambda} =
\lim\limits_{d\rightarrow +\infty}
\prod\limits_{\kappa=0}^d \Big(\sum\limits_{\ell =0}^\infty
\frac{(i^{1-\kappa} t \gamma_\kappa)^\ell}{\ell!} (i\lambda)^{\kappa \ell}\Big)\, e^{i\lambda z} \\
= \lim\limits_{d\rightarrow +\infty}
\prod\limits_{\kappa=0}^d \Big(\sum\limits_{\ell =0}^\infty
\frac{(i^{1-\kappa} t \gamma_\kappa)^\ell}{\ell!} (d/dW)^{\kappa \ell}\Big)\Big) (e^{i \lambda (\cdot)})(z) = \D_E(t) (e^{i\lambda (\cdot )})(z).
\end{multline}
One requires the following lemma in order to justify the formal relations \eqref{sect3-eq11}.

\begin{lemma}\label{sect3-lem1}
When $\rho=+\infty$, the convolution operator $\D_E(t)$ acts continuously locally uniformly with respect to $t\in \R$ from $A_1(\C)$ into itself. When $\rho \in ]0,+\infty[$ it acts continuously locally uniformly with respect to $t\in \R$ from the space
$$
\{f\in A_1(\C)\,;\, \forall\, \varepsilon>0,\
\exists\, C_\varepsilon>0\ {\rm such\ that}\ |f(W)|\leq C_\varepsilon
e^{(\rho-\varepsilon)|W|}\Big\} =\lim_{\longleftarrow} A_1^{B_{\rho,n}}(\C)
$$
(where $\{B_{\rho,n}\}_{n\geq 1}$ is a strictly increasing sequence  converging to $\rho$) into itself.
\end{lemma}

\begin{proof}
Suppose first that $\rho=+\infty$. Let $R>0$ and $K\subset [-R,R]\subset \R_t$ be a compact set. One recalls here that the radius of convergence of the power series
$\sum_{j\geq 0} \big(\sum_{\kappa\geq 0} |b_{j,\kappa}| R^\kappa\big) X^j$ equals $+\infty$.
Let $\gamma>0$, $\beta>0$ and $f=\sum_{\ell \geq 0} f_\ell\, W^\ell \in A_1^{\gamma,\beta}(\C)$.
One can check as in the proof of Lemma \ref{sect2-lem2} (compare to \eqref{sect2-eq4}) that, for any $t\in K$ and $j\in \N$,
$$
\sum\limits_{j=0}^\infty
\frac{(j+\ell)!}{\ell!} |b_j(t)| \, |f_{\ell +j}|
\leq \gamma \frac{\beta^\ell}{\ell!}\, \sum\limits_{j=0}^\infty
\Big(\sum\limits_{\kappa=0}^\infty |b_{j,\kappa}| R^\kappa\Big)
\beta^j
= K_{\D_E}(\beta,\gamma)\,
\frac{\beta^\ell}{\ell!}.
$$
This is indeed enough to conclude as in the proof of Lemma \ref{sect2-lem2} that $\D_E(t)$ acts continuously locally uniformly in $t$ from
$A_1(\C)$ into itself.
\hfil\break
Consider now the case where $\rho\in ]0,+\infty[$. For any $R>0$, the radius of convergence of the power series $\sum_{j\geq 0} \big(\sum_{\kappa\geq 0} |b_{j,\kappa}|\, R^\kappa\big) X^j$
is now at least equal to $\rho$. Repeating the preceeding argument (but taking now $\beta \leq \rho-\varepsilon$ for some $\varepsilon >0$ arbitrary small), one concludes that $\D_E(t)$ acts continuously locally uniformly in $t$ from $\displaystyle{\lim_{\longleftarrow} A_1^{B_{\rho,n}}(\C)}$ into itself.
\end{proof}

\noindent
We can now state the last result of this section.

\begin{theorem}\label{sect3-thm4}
Let $E=\sum_{\kappa=0}^\infty \gamma_\kappa X^\kappa \in \C[[X]]$ be a power series with radius of convergence $\rho\in ]2,+\infty]$. Then
$\check E := \sum_{\kappa =0}^\infty (-i)^{\kappa +1}\gamma_\kappa D^\kappa$ acts continuously from $\displaystyle{\lim_{\longleftarrow} A_1^{B_{\rho,n}}(\C)}$ into itself. For any $t\in \R$ and $a\in \R$ with $1<|a|<\rho-1$, the generalized Fourier sequence
$$
\Big\{x \in \R \longmapsto \psi_{E,N} (t,x,a) = \sum\limits_{j=0}^N C_j(N,a)\,
e^{i E(1-2j/N)t} e^{ix(1-2j/N)}\Big\}_{N\geq 1}
$$
is superoscillating. Moreover, for any such $a$ and $(\mu,\nu)\in \N^2$, the sequence of functions
\begin{multline}\label{sect3-eq12}
\frac{\partial^{\mu +\nu}}{\partial t^\mu
\partial x^\nu} (\psi_{E,N}(t,x,a)) =
(-i)^\mu \Big(\check E^{\odot^{\mu}}
\odot (d/dW)^\nu\Big)\big(\psi_{p,N}(t,\cdot,a)\big)(x)\\
= (-i)^\mu
\Big(\check E^{\odot^{\mu}}
\odot (d/dW)^\nu \odot \D_E(t)\Big)\big(F_N(\cdot,a)\big)(x)
\end{multline}
converges then uniformly on any compact in $\R^2_{t,x}$ to the fonction
$$
(t,x) \longmapsto (-i)^\mu
\Big(\check E^{\odot^{\mu}}
\odot (d/dW)^\nu \odot \D_E(t)\Big)\big(e^{it E(a)}
\, e^{ia(\cdot)}\big)(x).
$$
\end{theorem}

\begin{proof}
The fact that $\check E$ acts continuously from
$\lim_{\longleftarrow} A_1^{B_{\rho,n}}(\C_W)$ into itself follows from Lemma \ref{sect3-lem1}, considering just
$\check E$ (independent of the parameter $t$) instead of $\D_E(t)$.
For any $\lambda\in \R$ with $|\lambda|<\rho$, the operator
$\check E$ then acts on $e^{i\lambda(\cdot)}$ and it is immediate to check that for any $t\in \R$
\begin{equation}\label{sect3-eq13}
\forall\, (t,x)\in \R^2,\
\Big[\Big(i\, \frac{\partial}{\partial t} - \check E\Big)
\big(e^{it E(\lambda)}\, e^{i\lambda W}\big)\Big]_{W=x} = 0~;
\end{equation}
moreover $\big[(t,x) \mapsto e^{it E(\lambda)}\, e^{i\lambda x}\big]_{t=0}$ is $x\mapsto e^{i\lambda x}$. Therefore, for any $a\in \R$ and $N\in \N^*$, one has by linearity (since $\rho>1$)
\begin{equation}\label{sect3-eq14}
\forall\, (t,x)\in \R^2,\
\Big[\Big(i\, \frac{\partial}{\partial t} - \check E\Big)
\Big(\sum\limits_{j=0}^N C_j(N,a) e^{i E(1-2j/N)t}
e^{i(1-2j/N) W}\Big)\Big]_{W=x} = 0.
\end{equation}
Lemma \eqref{sect3-lem1} (applied this time with $\D_E(t)$), combined with Theorem \ref{sect2-thm1} and the estimates in the first line of
\eqref{sect2-eq9} in Lemma \ref{sect2-lem5}, imply that as soon as one has $|a|<\rho-1$
the sequence $\{z\in \C \mapsto \D_E(t)(F(\cdot,a))(z)\}_{N\geq 1}$ converges (locally uniformly with respect to the parameter $t$) to $z\mapsto e^{iaz}$ in
$A_1(\C)$. The last assertion in the particular case $\mu=\nu=0$ follows. The first equality in \eqref{sect3-eq12} comes from the identity \eqref{sect3-eq14}, while the second one comes
from \eqref{sect3-eq11} (as justified by Lemma \ref{sect3-lem1}).
The last assertion of the theorem when
$\mu,\nu$ are arbitrary is then a consequence of the continuity of
$d/dz$ from $\underset{\lim}{\longleftarrow} A_1^{B_{\rho,n}}(\C)$ into itself.
The superoscillating character of the sequence $\{\psi_{P,N}(t,\cdot,a)\}_{N\geq 1}$ follows from Definition \ref{sect3-def2}.
\end{proof}

\begin{remark}
{\rm When $E\in \R[[X]]$, $1<|a|<\rho-1$ and $\sup_{[-1,1]}(E)\leq 1 < |E(a)|$, the generalized Fourier sequence
$$
\Big\{(t,x) \in \R^2 \longmapsto \psi_{E,N} (t,x,a) = \sum\limits_{j=0}^N C_j(N,a)\,
e^{i E(1-2j/N)t} e^{ix(1-2j/N)}\Big\}_{N\geq 1}
$$
is also superoscillating, this time according to Definition \ref{sect3-def4} (with
$P_\infty(T,X)=TX$, $\kappa(\infty) = E(a)$ and $k(\infty) = a$).}
\end{remark}

\section{Regularization of formal Fresnel-type integrals}\label{sect4}

In order to settle from the mathematical point of view the approach to non-absolutely convergent integrals
 on the half-line $\R^{+*}$ or the whole real line $\R$ through the so-called principle of {\it regularization} that we will invoke in the remaining sections \ref{sect5} and \ref{sect6} (with respect to supershift considerations related to Schr\"odinger equations with specific potentials), we need to explain what regularization of formal Fresnel-type integrals on $\R^{+*}$ or $\R$ means.
\vskip 2mm
\noindent
Suppose that $\EuFrak T$ is a set of parameters. Let
$G~: (\EuFrak t,Z)\in \EuFrak t \times \C  \longmapsto G(\EuFrak t,Z)$
be a function which is entire as a function of $Z$ for each $\EuFrak t \in \EuFrak T$ fixed. Let also $\phi$ be a non-vanishing real function on $\EuFrak T$ that will play the role of a {\it phase function}. Let finally $\chi$ be a real number such that $\chi >-1$. In order to give a meaning to the formal integral
\begin{equation}\label{sect4-eq1}
\int_0^\infty (x')^\chi\, e^{-i\phi(\EuFrak t) (x')^2}\, G(\EuFrak t,x')\, dx'\quad (\chi >-1)
\end{equation}
we distinguish the cases where $\phi(\EuFrak t)>0$ and $\phi(\EuFrak t)<0$. In the first case ($\phi(\EuFrak t)>0$), we rewrite this (for the moment formal)  expression \eqref{sect4-eq1} as
\begin{multline}\label{sect4-eq2}
\int_0^\infty (x')^\chi\, e^{-i\phi(\EuFrak t) (x')^2}\, G(\EuFrak t,x')\, dx' = e^{-i(\chi+1)\pi/4}\, \int_{\R^{+*}\, e^{i\pi/4}} Z^\chi\, e^{-\phi(\EuFrak t)\, Z^2}\,
G(\EuFrak t, e^{-i\pi/4} Z)\, dZ \\=
\int_{\R^{+*}\, e^{i\pi/4}} Z^\chi\, e^{-\phi(\EuFrak t)\, Z^2}
\, F_+(\EuFrak t,Z)\, dZ
\end{multline}
with $F_+(\EuFrak t,Z) := e^{-i(\chi+1)\pi/4} G(\EuFrak t,e^{-i\pi/4} Z)$ for any
$\EuFrak t \in \EuFrak T$ and $Z\in \C$. In the second case ($\phi(\EuFrak t)<0$), we rewrite it as
\begin{multline}\label{sect4-eq3}
\int_0^\infty (x')^\chi\, e^{-i\phi(\EuFrak t) (x')^2}\, G(\EuFrak t,x')\, dx' =
e^{i(\chi+1)\pi/4}\, \int_{\R^{+*}\, e^{-i\pi/4}} Z^\chi\, e^{\phi(\EuFrak t)\, Z^2}\, G(\EuFrak t, e^{i\pi/4} Z)\, dZ \\
= \int_{\R^{+*}\, e^{-i\pi/4}} Z^\chi
e^{\phi(t)\, Z^2}\, F_-(\EuFrak t,Z)\, dZ
\end{multline}
with $F_-(\EuFrak t,Z) := e^{i(\chi+1)\pi/4} G(\EuFrak t,e^{i\pi/4} Z)$ for any
$\EuFrak t \in \EuFrak T$ and $Z\in \C$.
The following elementary lemma will reveal to be essential.

\begin{lemma}\label{sect4-lem1}
Let $\EuFrak T,\phi,\chi$ as above and
$F~: \EuFrak T \times \C \longrightarrow \C$
be a function with is entire in the complex variable $Z$ and satisfies the growth estimates
\begin{equation}\label{sect4-eq4}
\forall\, \varepsilon >0,\
\sup\limits_{\EuFrak t\in \EuFrak T,Z\in \C}
\big(|F(\EuFrak t,Z)|\, \exp(-\varepsilon |Z|^{\check p})\big) <+\infty
\end{equation}
for some $\check p\in ]1,2]$, that is $F(t,\cdot)\in A_{\check p,0}(\C)$ uniformly in $\EuFrak t$. Then, for any $u=e^{i\theta}$ with
$\theta \in ]-\pi/4,\pi/4[$, the integral
\begin{equation}\label{sect4-eq5}
\int_{\R^{+*}\, u} Z^\chi\, e^{- |\phi(\EuFrak t)|\, Z^2}\, F(\EuFrak t,Z)\, dZ
\end{equation}
is absolutely convergent and remains independent of $u$~; it equals in particular its value for $u=1$.
\end{lemma}

\begin{proof}
The absolute convergence follows from the estimates \eqref{sect4-eq4}, together with the fact that if $u=e^{i\theta}$, ${\rm Re}((t u)^2) =
t^2 \cos (2\theta)>0$ for $t>0$. The fact that the integrals do not depend of $u$ follows from residue theorem (applied on the oriented boundary of conic sectors with apex at the origin).
\end{proof}

\noindent
In view of this lemma, the regularization of an integral of the Fresnel-type such as \eqref{sect4-eq1} consists in the successive two operations~:
\begin{enumerate}
\item first transform the formal expression \eqref{sect4-eq1} into one of the representations
\eqref{sect4-eq2} or \eqref{sect4-eq3} according to
${\rm sign}(\phi(\EuFrak t))$~;
\item then invoke Lemma \ref{sect4-lem1} (provided the required hypothesis are satisfied) and consider then the regularization of \eqref{sect4-eq1} as $\int_0^\infty Z^\chi e^{-\phi(\EuFrak t) Z^2}
F_+(\EuFrak t,Z)\, dZ$ when $\phi(\EuFrak t)>0$
or $\int_0^\infty Z^\chi e^{\phi(\EuFrak t) Z^2}
F_-(\EuFrak t,Z)\, dZ$ when $\phi(\EuFrak t)<0$.
\end{enumerate}

\begin{remark}\label{sect4-rem1}
{\rm
In order to give a meaning (if possible of course) to the formal integral expression
\begin{equation}\label{sect4-eq6}
\int_\R |x'|^\chi\, e^{-i\phi(\EuFrak t) (x')^2}\, G(\EuFrak t,x')\, dx',
\end{equation}
one splits it as
$$
\int_0^\infty (x')^\chi \, e^{-i\phi(\EuFrak t) (x')^2}\, G(\EuFrak t,x')\, dx' +
\int_0^\infty (x')^\chi\, e^{-i\phi(\EuFrak t) (x')^2}\, G(\EuFrak t,-x')\, dx'
$$
and proceed as above for the two formal expressions involved into this formal decomposition.
}
\end{remark}

\noindent
It is immediate to compare this approach to regularization to the alternative following one.

\begin{proposition}\label{sect4-prop1} Let $G \in A_{2,0}(\C)$ and $\chi >-1$. Then, for all $\varpi \in \R^*$
$$
\lim\limits_{\varepsilon \rightarrow 0} \int_0^\infty
(x')^\chi\, e^{i\, \varpi (x')^2}\, e^{-\varepsilon (x')^2}\, G(x')\, dx'
$$
exists and coincides with the integral regularized under the approach described above.
\end{proposition}

\begin{proof}
It is enough to prove the result when $\varpi = \pm 1$ since one reduces to one of these two cases up to a homothety on the real half line. One has
\begin{eqnarray*}
\int_0^\infty (x')^\chi\, e^{-\varepsilon (x')^2} e^{-i\, (x')^2}
\, G(x')\, dx' &=&
\int_{e^{i\pi/4} \R^{+*}} Z^\gamma\, e^{- (1-i\, \varepsilon)\, Z^2}\, F_+(Z)\, dZ
\\
\int_0^\infty (x')^\chi\, e^{-\varepsilon (x')^2} e^{i\, (x')^2}
\, G(x')\, dx' &=&
\int_{e^{-i\pi/4} \R^{+*}} Z^\chi\, e^{- (1 + i\, \varepsilon)\, Z^2}\, F_-(Z)\, dZ,
\end{eqnarray*}
where $F_+(Z) = e^{-i(1+\chi)\pi/4} F(e^{-i\pi/4} Z)$ and $F_-(Z) = e^{i(1+\chi)\pi/4} F(e^{i\pi/4} Z)$. Let $\rho_\varepsilon = \sqrt{1+\varepsilon^2}$, and
$\xi_{\varepsilon} = {\rm arg}_{[0,\pi/2[} \sqrt{1+ i\varepsilon}$.
One has then
\begin{equation}\label{sect4-eq7}
\begin{split}
& \int_0^\infty (x')^\gamma\, e^{-\varepsilon (x')^2} e^{-i\, (x')^2}
\, G(x')\, dx' = \Big(\frac{e^{i\xi_\varepsilon}}{\sqrt{\rho_\varepsilon}}\Big)^{1+\chi}\,
\int_{e^{i(\pi/4-\xi_\varepsilon)}\, \R^{+*}}
Z^\chi\, e^{-Z^2} F^+(e^{i\xi_\varepsilon}\, Z/\sqrt{\rho_\varepsilon})\, dZ
\\
&
\int_0^\infty (x')^\chi\, e^{-\varepsilon (x')^2} e^{i\, (x')^2}\, G(x') \, dx'=
\Big(\frac{e^{- i\xi_\varepsilon}}{\sqrt{\rho_\varepsilon}}\Big)^{1+\chi}\,
\int_{e^{-i(\pi/4 -\xi_\varepsilon)}\, \R^{+*}}
Z^\chi\, e^{-Z^2}\, F_-(e^{-i\xi_\varepsilon}\, Z/\sqrt{\rho_\varepsilon})\, dZ.
\end{split}
\end{equation}
In the two integrals on the right-hand side of the equalities \eqref{sect4-eq7}, the integration contour can be replaced by
the half-line $\R^{+*}$ as a consequence of Lemma \ref{sect4-lem1}.
It is then possible to take the limit when $\varepsilon$ tends to
$0$. Lebesgue's domination theorem then applies and since
$\rho_\varepsilon$ tends to $1$ and $\xi_\varepsilon$ to
$0$, one gets
\begin{eqnarray*}
\lim\limits_{\varepsilon \rightarrow 0} \int_0^\infty (x')^\chi\, e^{-\varepsilon (x')^2} e^{-i\, (x')^2}
\, G(y)\, dy &=&
\int_0^\infty Z^\chi\, e^{- Z^2}\, F_+(Z)\, dZ
\\
\lim\limits_{\varepsilon \rightarrow 0}
\int_0^\infty (x')^\chi\, e^{-\varepsilon (x')^2} e^{i\, (x')^2}
\, G(x')\, dx' &=&
\int_0^\infty Z^\chi\, e^{- Z^2}\, F_-(Z)\, dZ.
\end{eqnarray*}
This concludes the proof of the Proposition.
\end{proof}

\section{Fresnel-type integral operators}\label{sect5}

\subsection{Continuity on $A_1(\C)$ of Fresnel-type integral operators}\label{sect5-1}

Let $\EuFrak T$ be a set of parameters and $\EuFrak t \in \EuFrak T\mapsto \D(\EuFrak t,Z)$ (as in the statement of Lemma \ref{sect2-lem3}) be a differential operator-valued map
$$
\EuFrak t \in \EuFrak T \longmapsto \D(\EuFrak t,Z) = \sum\limits_{j=0}^\infty b_j(\EuFrak t,Z) \Big(\frac{d}{dZ}\Big)^j
$$
(with $b_j~: T\times \C \rightarrow \C$, holomorphic in $Z$ for $j\in \N$) such that
\begin{equation}\label{sect5-eq1}
\forall\, \varepsilon >0,\
\sup\limits_{\EuFrak t \in \EuFrak T, (Z,W)\in \C^2}
\Big(\Big(\sum\limits_{j=0}^\infty |b_j(\EuFrak t,Z)|\ |W|^j\Big)\, \exp (- \varepsilon\, |Z|^{\check p} - B\, |W|^{p})\Big) = A^{(\varepsilon)} < +\infty
\end{equation}
for some $\check p\in ]1,2], p \geq 1$ and $B\geq 0$.
Let also $\phi$ be a non-vanishing real function on $\EuFrak T$ and $\chi>-1$. It follows from the estimates \eqref{sect5-eq1}, together with Lemma \ref{sect4-lem1}, that the regularization approach described in section \ref{sect4} allows to define the operator
\begin{equation}\label{sect5-eq2}
\EuFrak t \longmapsto \int_{0}^\infty Z^\chi\, e^{-i \phi(\EuFrak t)\, Z^2}
\sum\limits_{j=0}^\infty b_j(\EuFrak t,Z) \Big(\frac{d}{dZ}\Big)^j (\cdot )\, dZ.
\end{equation}
One needs to consider for the moment these operators as acting on entire functions of the complex variable $Z$. For $\alpha \in \C$, let also $H_\alpha$ be the dilation operator $H_\alpha~: f \mapsto f(\alpha (\cdot))$ acting on such functions. The symbol $\odot$ still stands for the composition of operators. The discussion is with respect to the sign of $\phi(\EuFrak t)$.
\begin{itemize}
\item
When $\phi(\EuFrak t)>0$,
\begin{multline}\label{sect5-eq3}
\int_{0}^\infty Z^\chi\, e^{-i \phi(\EuFrak t)\, Z^2} \Big(
\sum\limits_{j=0}^\infty b_j(\EuFrak t,Z) \Big(\frac{d}{dZ}\Big)^j (\cdot )\Big)\, dZ \\
= e^{-i(1+\chi) \pi/4}
\int_0^\infty
y^\chi\, e^{-\phi(\EuFrak t)\, y^2} \, \Big(
\sum\limits_{j=0}^\infty b_j(\EuFrak t, e^{-i\pi/4} Z)\,
\Big( e^{ij \pi/4} \, \Big(\frac{d}{dZ}\Big)^j \odot H_{e^{-i\pi/4}}\Big)\, (\cdot)\Big)(y)\, dy.
\end{multline}
\item
When $\phi(\EuFrak t)<0$,
\begin{multline}\label{sect5-eq4}
\int_{0}^\infty Z^\chi\, e^{-i \phi(\EuFrak t)\, Z^2}
\Big(\sum\limits_{j=0}^\infty b_j(\EuFrak t,Z) \Big(\frac{d}{dZ}\Big)^j (\cdot )\Big)\, dZ \\
= e^{i(1+\chi)\pi/4}
\int_0^\infty y^\chi\,
e^{\phi(\EuFrak t)\, y^2}
\, \Big(\sum\limits_{j=0}^\infty b_j(\EuFrak t, e^{i\pi/4}\, Z)\,
\Big( e^{-ij \pi/4} \, \Big(\frac{d}{dZ}\Big)^j \odot  H_{e^{i\pi/4}}\Big)\, (\cdot)\Big)(y)\, dy.
\end{multline}
\end{itemize}

\begin{theorem}\label{sect5-thm1}
Suppose that the parameter space $\EuFrak T$ is a topological space and that $\phi$ is continuous. Consider functions
$B_j~: \EuFrak T \times \C \times \C\rightarrow \C$ ($j\in \N$) which are entire in the two complex entries and such that
\begin{multline}\label{sect5-eq5}
\forall\, \varepsilon >0,\ \exists\, A^{(\varepsilon)}, B^{(\varepsilon)} \geq 0\ {\it such\ that}\ \forall\, \EuFrak t\in \EuFrak T,\
\forall Z\in \C,\forall\, \check Z\in \C,\ \forall\, W\in \C,\\
\sum\limits_{j=0}^\infty |B_j(\EuFrak t,Z,\check Z)|\ |W|^j \leq
A^{(\varepsilon)} \, e^{\varepsilon\, |Z|^{\check p} + B^{(\varepsilon)} |\check Z|^{\check p} + B\, |W|^{p}}
\end{multline}
for some $p\geq 1$, $\check p\in ]1,2]$, and $B\geq 0$.
Then the operator
$$
\int_{0}^\infty Z^\chi\, e^{-i \phi(\EuFrak t)\, Z^2} \Big(
\sum\limits_{j=0}^\infty B_j(\EuFrak t,Z,\check Z) \Big(\frac{d}{dZ}\Big)^j (\cdot )\Big)\, dZ
$$
(understood through the process of regularization as described above) acts continuously locally uniformly in $\EuFrak t$ from $A_1(\C)$ into $A_{\check p}(\C)$.
\end{theorem}

\begin{proof}
It is enough to consider $\EuFrak T$ as a neighborhood of a point
$\EuFrak t_0$ in which $\phi(\EuFrak t)\geq \varepsilon_0>0$ (since $\phi$ is continuous).
Let $\EuFrak f = \{Z\mapsto f_N(Z)\}_{N\geq 1}$ be a sequence of elements
in $A_1(\C)$ that converges towards $0$ in $A_1(\C)$, which means (see Proposition \ref{sect2-prop1})
that all $f_N$ belong to some $A_1^{C,b}(\C)$ for some constants $C,b>0$ independent on $N$ (namely $f_N=\sum_\ell a_{N,\ell} Z^\ell$ with $|a_{N,\ell}|\leq C\, b^\ell/\ell!$ for any $\ell\in \N$).
It is clear that the operator
$$
\sum\limits_{j=0}^\infty B_j(\EuFrak t, e^{-i\pi/4} Z,\check Z)\,
\Big( e^{ij \pi/4} \, \Big(\frac{d}{dZ}\Big)^j \odot H_{e^{-i\pi/4}}\Big)
$$
involved in the integrand of \eqref{sect5-eq3}
is governed by estimates of the form \eqref{sect5-eq5}. It follows then from Lemma \ref{sect2-lem3}, taking into account estimates \eqref{sect5-eq5}, that for each $N\in \N^*$ the function
\begin{multline*}
H(f_N)~: (\EuFrak t, Z,\check Z) \in \EuFrak T \times \C \times \C  \\ \longmapsto \sum\limits_{j=0}^\infty B_j(\EuFrak t, e^{-i\pi/4} Z,\check Z)\,
\Big( e^{ij \pi/4} \, \Big(\frac{d}{dZ}\Big)^j \odot H_{e^{-i\pi/4}}\Big)(f_N)(Z)
\end{multline*}
is such that for each $\varepsilon>0$, there exists $\widetilde A^{(\varepsilon)}\geq 0$ (depending on $\EuFrak T$, $A^{(\varepsilon)}$, the $B_j$, $b$ and $C$, but not on the $N$) such that
$$
\forall\, (\EuFrak t,Z,\check Z)\in \EuFrak T \times \C \times \C,\quad
|H(f_N)(\EuFrak t,Z,\check Z)|\leq
\widetilde A^{(\varepsilon)}\, e^{\epsilon\, |Z|^{\check p} + B^{(\varepsilon)} |\check Z|^{\check p}}.
$$
Take in particular $\varepsilon < \varepsilon_0$. Then the function
$$
\check Z \in \C \longmapsto
\int_0^\infty y^\chi\,
e^{-\phi(t) y^2}
H(f_N) (\EuFrak t,y,\check Z)\, dy
$$
is in $A_{\check p,0}(\C)$ since it is estimated as
$$
\Big|
\int_0^\infty y^\chi\,
e^{-\phi(t) y^2}
H(f_N) (\EuFrak t,y,\check Z)\, dy\Big|
\leq
\widetilde A_{(\varepsilon)}\, \Big(\int_0^\infty
y^\chi\, e^{-\varepsilon_0 y^2} e^{\varepsilon y^{\check p}}\, dy\Big)\, e^{B^{(\varepsilon)} |\check Z|^{\check p}}\quad \forall\, \check Z \in \C
$$
(remember that $\check p \in ]1,2]$). It remains to show that the sequence
$$
\Big\{\check Z \longmapsto \int_0^\infty y^\chi\, e^{-\phi(t) y^2}
H(f_N) (\EuFrak t,y,\check Z)\, dy\Big\}_{N\geq 1}
$$
converges to $0$ in $A_{\check p,0}(\C)$. It is enough (see Proposition \ref{sect2-prop1}) to prove that it converges to $0$ uniformly on any closed disk
$\overline{D(0,r)}$ in $\C$.
Fix $\varepsilon < \varepsilon_0$ and $\eta>0$. Choose then
$R_\eta >>1 $ such that
\begin{multline*}
\forall\, N\in \N,\quad
\Big|\int_{R_\eta}^\infty
 y^\chi\, e^{-\phi(t) y^2}
H(f_N) (\EuFrak t,y,\check Z)\, dy\Big| \\
\leq \widetilde A^{(\varepsilon)}\, \Big(\int_0^\infty y^\chi\,
e^{-\varepsilon_0 y^2} e^{\varepsilon y^{\check p}}\, dy\Big)\,
e^{B^{(\varepsilon)}\, |\check Z|^{\check p}} \leq \eta\, e^{-
B^{(\varepsilon)} r^{\check p}}\,
e^{B^{(\varepsilon)} |\check Z|^{\check p}} \leq \eta.
\end{multline*}
On $[0,R_\eta]$, one uses the uniform convergence of $\EuFrak f$ towards
$0$ on any compact set, hence of $H[\EuFrak f]$ on any compact set, to conclude that for $N\geq N_\eta >>1$, one has
$$
\Big|\int_{0}^{R_\eta}
 y^\chi\, e^{-\phi(t) y^2}
H(f_N) (\EuFrak t,y,\check Z)\, dy\Big|
\leq \eta \quad \forall\, \check Z\in \overline{D(0,r)}.
$$
Note that our estimates show that the convergence towards $0$ in $A_{\check p,0}(\C)$ thus obtained is uniform in $\EuFrak t\in \EuFrak T$.
\end{proof}

\subsection{Superoscillations and supershifts}

Consider the Schr\"odinger equation
\begin{equation}\label{sect5-eq6}
i\, \frac{\partial \psi}{\partial t} (t,x) = \big(\mathscr H(x) (\psi)\big)(t,x)
\end{equation}
where $\mathscr H$ denotes the Hamiltonian operator attached to the physical system which is under consideration. Suppose that
$\boldsymbol Y = \{x\mapsto Y_N(x)\}_{N\geq 1}$ is a superoscillating sequence.
Since
$$
\Big(i\,
\frac{\partial}{\partial t} - \mathscr H(x)\Big)(\psi)(t,x) =0,\quad
\big[\psi(t,x)]_{t=0} = Y(x)
$$
is a Cauchy-Kowalevski problem (assuming that $x$ lies in some open set
$U\subset \R$ where the Hamiltonian operator is regular),
any entry $x\in U\mapsto Y_N(x)$ evolves in a unique way from $t=0$ towards $t>0$ as $(t,x)\mapsto \psi_N(t,x)$. Assume in addition that $x$ lies in the maximal superoscillation domain $U_{\rm max}^{\rm suposc}$~; the limit function $x\in U \cap U_{\rm max}^{\rm sosc} \mapsto
Y_\infty(x)$ then also evolves from $U\cap U_{\rm max}^{\rm sosc}$ into some function $(t,x) \mapsto \psi_\infty(t,x)$.
\vskip 2mm
\noindent
A natural question then occurs. As long as the evolution persists (let say for $t\in [0,T]$), is it true that the sequence
$\{x\in U \mapsto \psi_N(t,x)\}_{N\geq 1}$ is such that its restriction to $U\cap U_{\rm max}^{\rm sosc}$ converges (uniformly on any compact subset of $U\cap U_{\rm max}^{\rm sosc}$) to $x\mapsto \psi_\infty(t,x)$? If this is the case, one will say that the superoscillating character of the sequence $\boldsymbol Y$ persists in time through the Schr\"odinger evolution operator $\partial/\partial t - \mathscr H$ which is here considered.
\vskip 2mm
\noindent
In order to formulate such question in a different way, let us now consider the $(t,x)$ domain
$[0,T] \times (U\cap U_{\rm max}^{\rm sosc}) =
\EuFrak T \times (U\cap U_{\rm max}^{\rm sosc}) = \mathscr T$ as a parameter set and focus on the map $\lambda \in \R \longmapsto \varphi_\lambda$, where $\varphi_\lambda~: \mathscr T \rightarrow \R$ is evolved to $[0,T] \times U$ (through the Schr\"odinger operator) from the initial datum $x\in U \mapsto e^{i\lambda x}$, then restricted to the parameter set $\mathscr T$. Previous considerations lead to the following definition, which is inspired by Definition \ref{sect3-def2}.

\begin{definition}\label{sect5-def1} Let $\mathscr T$ be a locally compact topological space and $\mathscr F = \{\varphi_\lambda\,;\, \lambda \in \R\}$ be a family of $\C$-valued functions on $\mathscr T$ indexed by $\R$. A sequence $\boldsymbol \psi = \{\tau \in \mathscr T \mapsto \psi_N(\tau)\}_{N\geq 1}$ of
$\C$-valued functions on $\mathscr T$ is called a {\rm supershift for the family $\mathscr F$} (or {\rm $\mathscr F$ admits $\boldsymbol \psi$ as a supershift}) if
\begin{itemize}
\item
any entry $\psi_N$ is of the form
$\psi_N = \sum_{j=0}^N C_j(N) \, \varphi_{k_j(N)}$
with $|k_j(N)|\leq 1$ for any $N\in \N^*$ and $0\leq j\leq N$;
\item there exists an open subset $\mathscr U^{\rm ssh}$ of $\mathscr T$
called a {\rm $\mathscr F$-supershift domain} such that the sequence
$\{\tau \in \mathscr U^{\rm ssh} \mapsto \psi_N(\tau)\}$ converges
locally uniformly towards the restriction to $\mathscr U^{\rm ssh}$
of a function $\psi_\infty$ which is a $\C$-finite linear combination of
elements in $\mathscr F$ of the form $\varphi_{\nu k(\infty)}$ with $\nu\in \Z^*$, where $k(\infty) \in \R \setminus [-1,1]$.
\end{itemize}
\end{definition}

\vskip 2mm
\noindent 
\begin{example}\label{sect5-expl1}{\rm
\vskip 1mm
\noindent
\begin{enumerate}
\item If $\mathscr T =\R$ and $\mathscr F$ denotes the family of
characters $x\in \R \mapsto e^{i\lambda x}$ indexed by the dual copy $\R_\lambda^\star$ of $\R_x$,
$\mathscr F$-supershifts are the complex superoscillating sequences (see Definition \ref{sect3-def2}).
\item Let $a\in \R\setminus [-1,1]$, $\mathscr T=\R^2_{t,x}$ and
$\mathscr F = \{\varphi_\lambda\,;\, \lambda \in \R^\star\}$ as defined
in Theorem \ref{sect3-thm2} or Theorem \ref{sect3-thm3}. For any
$a\in \R\setminus [-1,1]$, the sequence $\{(t,x) \mapsto \psi_{P,N}(t,x,a)\}_{N\geq 1}$ is a $\mathscr F$-supershift which admits $\R^2_{t,x} = \mathscr T$ as $\mathscr F$-supershift domain.
\end{enumerate}
}
\end{example}
\vskip 2mm
\noindent
When $\mathscr T$ is of the form $[0,T[\times U$, where $U$ is an open subset in $\R^{m-1}_x$ ($m\geq 2$) and $T\in ]0,+\infty]$, one can consider as well families $\mathscr F = \{\varphi_\lambda\,;\, \lambda \in \R\}$ of {\it $\C$-valued distributions in $\R \times U$ with support in $[0,T[\times U$}.
In order to define in this new context the notion of $\mathscr F$-supershift, one needs to keep the first clause in Definition \ref{sect5-def1} as it is and modify the second clause as follows~: ``{\it there exists an open subset $\mathscr U^{\rm ssh}= \mathscr V^{\rm ssh} \cap \mathscr T$ (where $\mathscr V$ is an open subset in $\R \times U$), called a {\rm $\mathscr F$-supershift domain},  such that the sequence
$\{(\psi_N(\tau))_{|\mathscr V^{\rm ssh}}\}_{N\geq 1}$ converges weakly
{\it in the sense of distributions} in $\mathscr V$ to the restriction to $\mathscr V$ of a {\it distribution}
$\psi_\infty\in \mathscr D'(\R \times U,\C)$ which is a $\C$-finite linear combination of
elements in $\mathscr F$ of the form $\varphi_{\nu k(\infty)}$ with $\nu\in \Z^*$, where $k(\infty) \in \R \setminus [-1,1]$}''.
\vskip 2mm
\noindent
One will need in section \ref{sect6} a further extension of this concept of $\mathscr F$-supershift to the case where $\mathscr T = [0,T[ \times U$, $U\subset \R^{m-1}$ with $m\geq 2$ as above, but elements $\varphi_\lambda\in \mathscr F$ are now
{\it hyperfunctions in $\R \times U$ with support in $\mathscr T$}.
The sequence $\{(\psi_N(\tau))_{|\mathscr V^{\rm ssh}}\}_{N\geq 1}$ needs in this case to converge still in the weak sense, but this time {\it in the sense of hyperfunctions} in $\mathscr V$, towards the restriction to $\mathscr V$ of the {\it hyperfunction} $\psi_\infty$. The notion of $\mathscr F$-supershift can thus be enlarged to families $\mathscr F =\{\varphi_\lambda\,;\, \lambda \in \R\}$ of {\it hyperfunctions} in $\R\times U$ with support in $\mathscr T$.

\subsection{The Schr\"odinger Cauchy problem with centrifugal potential}\label{sect5-2}

We will consider in this subsection the case where $U=\{x\in \R\,;\, x>0\}$ and the hamiltonian in \eqref{sect5-eq6} is $x\in U \mapsto \mathscr H(x) = - (\partial^2/\partial x^2)/2 + u/(2x^2)$, where $u$ denotes a real strictly positive physical constant. The corresponding Cauchy-Kowalevski problem (with $[0,+\infty[ \times U$ as phase space) is the {\it Schr\"odinger Cauchy problem with centrifugal potential}, see
\cite{centrifugal} for more references. For this Cauchy-Kowalevski problem, the analysis of the evolution $t\mapsto \psi(t,\cdot)$ of the solution $(t,x) \in [0,\infty[ \times U \mapsto \psi(t,x)$ from
an initial datum $x\in U\mapsto \psi(0,x)$ can be carried through thanks to the explicit form of the Green function $(t,x,x') \mapsto G(t,x,t'=0,x')$.
\vskip 2mm
\noindent
Let $\nu = \sqrt{1+4u}/2$ and the Bessel function $J_\nu$ defined in $\Omega := \C\setminus ]-\infty,0]$ as
\begin{multline}\label{sect5-eq7}
J_\nu~: z \in \Omega \longmapsto
\Big(\frac{z}{2}\Big)^\nu \sum\limits_{k=0}^\infty
\frac{(-1)^k}{\Gamma (k+1)\, \Gamma(\nu+k+1)}\, \Big(\frac{z}{2}\Big)^{2k} \\
= \Big(\frac{|z|}{2}\Big)^\nu
e^{i\, \nu\, {\arg}_{]-\pi,\pi[}(z)}\,
\sum\limits_{k=0}^\infty
\frac{(-1)^k}{\Gamma (k+1)\, \Gamma(\nu+k+1)}\, \Big(\frac{z}{2}\Big)^{2k} \\
= \Big(\frac{|z|}{2}\Big)^\nu
e^{i\, \nu\, {\arg}_{]-\pi,\pi[}(z)}\, E_\nu(z).
\end{multline}
Then the Green function $(t,x,x')\mapsto G(t,x,0,x')$ can be explicited in this case as
\begin{equation}\label{sect5-eq8}
G(t,x,0,x') = (-i)^{\nu+1}\,
\frac{\sqrt {xx'}}{t}\, \exp \Big(i \frac{x^2 + (x')^2}{2t}\Big)\, J_\nu \Big(\frac{xx'}{t}\Big)\quad (t>0,x,x'\in U)
\end{equation}
(see \cite{kempf1,SchulmanBOOK,Green}).

\begin{proposition}\label{sect5-prop1}
Let $\mathscr T = ]0,+\infty[\times ]0,+\infty[$, $\mathscr H:
x\in ]0,+\infty[ \mapsto -(\partial^2/\partial x^2 - u/x^2)/2$ for some physical constant $u>0$.
For any $\lambda \in \R$, the initial datum
$x\in ]0,+\infty[ \mapsto e^{i\lambda x}$ evolves through the
Cauchy-Kovalewski Schr\"odinger equation \eqref{sect5-eq6}
to a function $(t,x) \mapsto \varphi_\lambda (t,x)$ which is
$C^\infty$ in $\mathscr T$. For any $a\in \R \setminus [-1,1]$,
the family $\{\varphi_\lambda\,;\, \lambda \in \R\}$ admits as a
$\mathscr F$-supershift (in the sense of Definition \ref{sect5-def1}) the sequence
$$
\{(t,x)\in \mathscr T \mapsto \psi_N(t,x,a)\}_{N\geq 1} =
\Big\{\sum\limits_{j=0}^N C_j(N,a)\, \varphi_{1-2j/N}\Big\}
$$
with $\mathscr F$-supershift domain equal to $\mathscr T$.
Moreover, for any $(\mu,\nu)\in \N^2$, the sequence of functions
$$
\frac{\partial^{\mu+\nu}}{\partial t^\mu
\partial x^\nu} (\psi_N(t,x,a)) = \frac{1}{(2i)^{\mu}}
\Big(\Big(-\frac{\partial^2}{\partial x^2} + \frac{u}{x^2}\Big)^{\odot^\mu}
\odot \frac{\partial^{\nu}}{\partial x^\nu}\Big) (\psi_N(t,x,a))
$$
converges uniformly on any compact $K\subset\subset \mathscr T$ to the function
\begin{equation*}
(t,x) \in \mathscr T \longmapsto \frac{1}{(2i)^{\mu}}
\Big(\Big(-\frac{\partial^2}{\partial x^2} + \frac{u}{x^2}\Big)^{\odot^\mu}
\odot \frac{\partial^{\nu}}{\partial x^\nu}\Big)(\varphi_a(t,x)).
\end{equation*}
\end{proposition}

\begin{proof}
Let $\lambda \in \R$. The evolution of the initial datum
$x\in U \mapsto e^{i\lambda x}$ through the Schr\"odinger equation
\eqref{sect5-eq6} is explicited (for the moment formally)  thanks to the expression
\eqref{sect5-eq8} of the Green function as
\begin{equation}\label{sect5-eq9}
(t,x) \in \mathscr T
\longmapsto
\frac{(-i)^{\nu +1}}{2^\nu}\, e^{ix^2/(2t)}\,
\frac{x^{\nu+1/2}}{t^{\nu+1}}\,
\int_{0}^\infty
(x')^{\nu+1/2}\, e^{i (x')^2/(2t)}\, E_\nu \Big(\frac{xx'}{t}\Big) e^{i\lambda x'}\, dx'.
\end{equation}
For any $M\in \N$ such that $2M>\nu-1/2$ and any $y>0$, one has
\begin{multline*}
E_\nu(y) = \frac{1}{\sqrt \pi}
\, \Big(\frac{2}{y}\Big)^{\nu+1/2}\,
\Big(\cos \big(y -\nu \pi/4 -\pi/2\big) \Big( \sum\limits_{\kappa =0}^{M-1}
(-1)^\kappa \frac{a_{2\kappa}(\nu)}{y^{2\kappa}} + R_{2M}(\nu,y)\Big)
\\
+ \sin \big(y -\nu \pi/4 - \pi/2\big) \Big( \sum\limits_{\kappa =0}^{M-1}
(-1)^\kappa \frac{a_{2\kappa+1}(\nu)}{y^{2\kappa+1}} + R_{2M+1}(\nu,y)\Big)\Big)
\end{multline*}
with
$$
|R_{2M}(\nu,y)| < \frac{|a_{2M}(\nu)|}{y^{2M}},\quad
|R_{2M+1}(\nu,y)| < \frac{|a_{2M+1}(\nu)|}{y^{2M+1}},
$$
where
\begin{equation*}
a_k(\nu) = (-1)^k \frac{\cos (\pi \nu)}{\pi}
\, \frac{\Gamma (k+1/2 +\nu) \Gamma (k+1/2-\nu)}{2^k \Gamma(k+1)}\quad
\forall\, k\in \N
\end{equation*}
(see \cite[pp. 207-209]{Watson}). It follows from such developments, together with Proposition \ref{sect4-prop1}, that the integral in \eqref{sect5-eq9} exists for any $(t,x)\in \mathscr T$ as a semi-convergent integral (of the Fresnel-type), whose value coincides with the regularized integral described in section \ref{sect4}.
Set now
\begin{eqnarray*}
\EuFrak T = ]0,+\infty[, && \phi~: t\in \EuFrak T
\longmapsto -\frac{1}{2t} \in ]-\infty,0[\\
B_j~: (t,Z,\check Z) \in \EuFrak T \times \C^2
&& \longmapsto \begin{cases} E_\nu \Big(\displaystyle{\frac{Z \check Z}{t}}\Big)\
{\rm if}\ j=0 \\
0 \ {\rm if}\ j\in \N^*.
\end{cases}, \quad \chi := \nu + \frac{1}{2},
\end{eqnarray*}
in order to fit with the setting described in Theorem \ref{sect5-thm1}.
Since $E_\nu\in A_1(\C)$ and
$$
\frac{|Z\, \check Z|}{t}  =
\frac{1}{t}\times \varepsilon |Z| \times
\frac{|\check Z|}{\varepsilon} \leq \frac{1}{2t} \Big( \varepsilon^2\, |Z|^2 +
\frac{|\check Z|^2}{\varepsilon^2}\Big)\, \forall t>0,\ \forall\,
(Z,\check Z)\in \C^2
$$
the operator with order $0$ given as
$t \mapsto B_0(t,Z,\check Z) \, (d/dZ)^0$ satisfies the hypothesis of this theorem with $p=1$ and $\check p = 2$.
Then the operator
$$
\D(t) = \int_0^\infty Z^{\nu+1/2}\, e^{-i\, Z^2/(2t)}\, E_\nu\Big(
\frac{Z\cdot \check Z}{t}\Big)(\cdot )\, dZ
$$
acts continuously locally uniformly in $t \in ]0,+\infty[$ from $A_1(\C)$ into $A_2(\C)$. For any $\lambda \in \R$ and $t >0$, the function
$x\in ]0,+\infty[ \mapsto \varphi_\lambda (t,x)$ is $C^\infty$ because of its expression \eqref{sect5-eq9}. Moreover, when $a\in \R \setminus [-1,1]$, it follows from Theorem \ref{sect2-thm1} that the sequence
$$
\Big\{z\in \C \mapsto \D(t)\Big(\sum\limits_{j=0}^N C_j(N,a) e^{i(1-2j/N)(\cdot)}\Big)(z)\Big\}_{N\geq 1}
$$
converges in $A_2(\C)$ (locally uniformly with respect to $t>0$) to $z\mapsto \D(t)(e^{ia(\cdot)})$. One concludes then to the second assertion in the statement of the theorem. As for the last assertion, it follows from the fact that the action of $i\partial/\partial t$ and
$\mathscr H(x)$ coincide on solutions of \eqref{sect5-eq6}, together with the continuity of the operator $d/dz$ from $A_2(\C)$ into itself.
\end{proof}

\subsection{The Schr\"odinger Cauchy problem for the quantum harmonic oscillator}\label{sect5-3}

Let now $U=\R$ and the hamiltonian in \eqref{sect5-eq6} be $x\in \R \mapsto \mathscr H(x) = - (\partial^2/\partial x^2)/2 + x^2/2$. The corresponding Cauchy-Kowalevski problem (with $[0,+\infty[ \times \R$ as phase space) is the {\it Schr\"odinger Cauchy problem for the quantum harmonic oscillator}, see
\cite[\S 5.3, \S. 6.4]{acsst5} or \cite{harmonic} for more references.
In this case again, the Green function can be explicited and is therefore handable. It is the locally integrable function in
$]0,+\infty[\times \R \times \R$ defined as
\begin{multline}\label{sect5-eq10}
G(t,x,t'=0,x') =
\sqrt{\frac{1}{2i\pi \sin t}} \, e^{\displaystyle{
i \, \Big(\frac{(x^2+(x')^2)\cos t - 2 xx'}{2 \sin t}}\Big)} \\
= \Big(\sqrt{\frac{1}{2i\pi \sin t}}\,
e^{\displaystyle{i\, \frac{{\rm cotan}\, t}{2}\, x^2}}\Big)\,
e^{\displaystyle{i\, \frac{{\rm cotan}\, t}{2}\, (x')^2}}\, e^{ \displaystyle{- i \frac{xx'}{\sin t}}}\quad (t>0, x, x'\in \R).
\end{multline}

\begin{proposition}\label{sect5-prop2}
Let $\mathscr T = ]0,+\infty[\times \R$ and $\mathscr H~: x\in \R \mapsto -(\partial^2/\partial x^2 - x^2)/2$.
For any $\lambda \in \R$, the initial datum
$x\in \R \mapsto e^{i\lambda x}$ evolves through the
Cauchy-Kovalewski Schr\"odinger equation \eqref{sect5-eq6}
to a $\C$-valued distribution $\varphi_\lambda
\in \mathscr D'(\mathscr T,\C)$ with singular support
$\pi (2\N+1)/2 \times \R$. Let
$\mathscr U = \mathscr T \setminus (\pi (2\N+1)/2 \times \R)$. For any $a\in \R \setminus [-1,1]$, the family $\{(\varphi_\lambda)_{|\mathscr U}\,;\, \lambda \in \R\}$, considered as a family of functions, admits as a
$\mathscr F$-supershift (in the sense of Definition \ref{sect5-def1}) the sequence
$$
\{(t,x)\in \mathscr U \mapsto \psi_N(t,x,a)\}_{N\geq 1} =
\Big\{\sum\limits_{j=0}^N C_j(N,a)\, (\varphi_{1-2j/N})_{|\mathscr U}\Big\}_{N\geq 1}
$$
with $\mathscr F$-supershift domain equal to $\mathscr U$.
Moreover, for any $(\mu,\nu)\in \N^2$, the sequence of functions
from $\mathscr U$ to $\C$
$$
\frac{\partial^{\mu+\nu}}{\partial t^\mu
\partial x^\nu} (\psi_N(t,x,a)) = \frac{1}{(2i)^{\mu}}
\Big(\Big(-\frac{\partial^2}{\partial x^2} + x^2 \Big)^{\odot^\mu}
\odot \frac{\partial^{\nu}}{\partial x^\nu}\Big) (\psi_N(t,x,a))
$$
converges uniformly on any compact $K\subset\subset \mathscr U$ to the function
\begin{equation*}
(t,x) \in \mathscr T \longmapsto \frac{1}{(2i)^{\mu}}
\Big(\Big(-\frac{\partial^2}{\partial x^2} + x^2 \Big)^{\odot^\mu}
\odot \frac{\partial^{\nu}}{\partial x^\nu}\Big)(\varphi_a(t,x)).
\end{equation*}
\end{proposition}

\begin{proof}
Consider the two (for the moment formal) operators
\begin{equation}\label{sect5-eq11}
t\in\, ]0,+\infty[\,\setminus\, \pi \N^*/2\, \longmapsto\, |\sin t|\, \int_0^\infty
e^{\displaystyle{i\, \frac{\sin 2t}{4}\, Z^2}}
e^{ \displaystyle{- i \varpi\, \,{\rm sign}\, (\sin (t))\, \check Z Z }}\circ H_{\varpi |\sin t|}\, (\cdot)\, dZ
\end{equation}
($\varpi=\pm 1$) which appear (after performing the change of variables
$Z\leftrightarrow |\sin t|\, Z$ on $[0,+\infty[$) in the splitting of
$$
t\in\, ]0,+\infty[\,\setminus\, \pi \N^*/2 \, \longmapsto \int_\R e^{\displaystyle{i\, \frac{{\rm cotan}\, t}{2}\, Z^2}}\,
e^{ \displaystyle{- i\, \check Z Z/\sin t}}\, (\cdot)\, dZ
$$
(see Remark \ref{sect4-rem1}).
Set now
\begin{eqnarray*}
\EuFrak T = ]0,+\infty[\, \setminus\, \pi \N^*/2, && \phi~: t\in \EuFrak T
\longmapsto - \frac{\sin (2t)}{4}\\
B_j~: (t,Z,\check Z) \in \EuFrak T \times \C^2
&& \longmapsto \begin{cases} \exp (-i\, \varpi\, {\rm sign}\, (\sin (t))\, Z \check Z)\odot
H_{\varpi |\sin t|}
\
{\rm if}\ j=0 \\
0 \ {\rm if}\ j\in \N^*.
\end{cases},\quad \chi=0
\end{eqnarray*}
($\varpi =\pm 1$) in order to fit with the setting described in Theorem \ref{sect5-thm1}.
As in the proof of Proposition \ref{sect5-prop1}, this theorem applies here and the two operators \eqref{sect5-eq11} act continuously from $A_1(\C)$ to $A_2(\C)$ (locally uniformly with respect to the parameter $t\in \EuFrak T$). Note again that the Fresnel-type integrals \eqref{sect5-eq11}, where $Z \mapsto e^{i\lambda Z}$ ($\lambda \in \R)$ is taken inside the bracket and $\check Z \in \R$, are semi-convergent and their values as semi-convergent integrals coincide with the values that are obtained by regularization as in section \ref{sect4}. In fact, in the case where $\check Z = x\in \R$
and $t\in \EuFrak T$, the value of
$$
\Big(\sqrt{\frac{1}{2i\pi \sin t}}\,
e^{\displaystyle{i\, \frac{{\rm cotan}\, t}{2}\, \check Z^2}}\Big)
\int_\R e^{\displaystyle{i\, \frac{{\rm cotan}\, t}{2}\, Z^2}}\,
e^{ \displaystyle{- i\, \check Z Z/\sin t}}\, (\cdot)\, dZ
$$
(understood as a regularized integral, see section \ref{sect4}, in particular Remark \ref{sect4-rem1}) equals
$$
(\cos t)^{-1/2} e^{- i \check Z^2\, \frac{{\rm tan}(t)}{2}}\,
e^{- i \lambda^2\, \frac{{\rm tan}(t)}{2}} \odot H_{1/\cos t}(e^{i\lambda (\cdot)})(\check Z)
$$
(see \cite[Proposition 5.3.1]{acsst5}).
Since $(t,x) \in ]0,+\infty[\times \R
\longmapsto (\cos t)^{-1/2} e^{- i x^2\, {\rm tan}(t)}\,
e^{- i \lambda^2\, {\rm tan}(t)}$ is a locally integrable function, the
initial datum $x\in \R \mapsto e^{i\lambda x}$ evolves through the Schr\"odinger equation \eqref{sect5-eq6} as a distribution $\varphi_\lambda$
(in fact defined by a locally integrable function).
Let $\D(t)$ the differential operator
$$
\D~: t\in ]0,\infty[\, \setminus\,
\pi\, \frac{2\N+1}{2} \longmapsto \sum\limits_{j=0}^\infty
\frac{1}{j!} \Big( i \frac{\sin 2t}{4}\Big)^j
(d/dW)^{2j}.
$$
Since
$$
(\cos t)^{-1/2} e^{- i \check Z^2\, \frac{{\rm tan}(t)}{2}}\,
e^{- i \lambda^2\, \frac{{\rm tan}(t)}{2}} \odot H_{1/\cos t}(e^{i\lambda (\cdot)})(\check Z) = (\cos t)^{-1/2} e^{- i \check Z^2\, \frac{{\rm tan}(t)}{2}}\, \D(t) (e^{i\lambda (\cdot)})(\check Z),
$$
and $\D$ acts continuously locally uniformly in $t$ from $A_1(\C)$ to $A_2(\C)$
thanks to Lemma \ref{sect2-lem2}, the sequence
$$
\Big\{\sum\limits_{j=0}^N C_j(N,a)\, (\varphi_{1-2j/N})_{|\mathscr U}\Big\}_{N\geq 1}
$$
is, for any $a\in \R\setminus [-1,1]$, a supershift for the family $\mathscr F = \{(\varphi_\lambda)_{|\mathscr U}\,;\, \lambda \in \R\}$ (with $\mathscr F$-supershift domain $\mathscr U$). The last assertion follows from the same argument than that used for the last assertion in Proposition \ref{sect5-prop1}.
\end{proof}

\section{Singularities in the quantum harmonic oscillator evolution}\label{sect6}

This section is the natural continuation of subsection \ref{sect5-3}.
We continue to investigate with respect to the notion of supershift the evolution of initial data $x\in \R \mapsto e^{i\lambda x}$, when $\lambda \in \R$, through the Cauchy-Schr\"odinger problem for the quantum harmonic oscillator and focus now on singularities.
In this section we keep the same notations as in Proposition \ref{sect5-prop2} and fix a point $(t_0,x_0)$ in $\mathscr T
\setminus \mathscr U$. We will just consider the case $t_0=\pi/2$ since
the situation is essentially identical at any point $((2k+1)\pi/2,x_0)$ with $k\in \N$ and $x_0\in \R$.
\vskip 2mm
\noindent
Let, for $\lambda \in \R$, $\varphi_\lambda \in \mathscr D'(\mathscr T,\C)$ be the distribution evolved from the initial datum $x\mapsto e^{i\lambda x}$ through the Schr\"odinger operator for the quantum oscillator problem \eqref{sect5-eq6} (with $\mathscr H~: x\in \R \mapsto (-\partial^2/\partial x^2 + x^2)/2$).
\vskip 2mm
\noindent
Let $\theta \in \mathscr D(\mathscr T,\C)$ be a test-function with support in a small neighborhood of $(\pi/2,x_0)$ and
$(t,x) \mapsto \xi(t,x) := \theta (t,x) \exp ((i x^2 {\rm cotan}\, t)/2)/\sqrt{2i\pi}$. One has (formally) for any $\lambda \in \R$,
\begin{multline}\label{sect6-eq1}
\langle \varphi_\lambda,\theta \rangle
= - \int_{\R^2}
\Big[
\int_\R e^{i\frac{\sin u}{2}\, Z^2} e^{-i \check Z Z/\sqrt{\cos u}} (e^{i\lambda (\cdot)})\, dZ\Big]_{\check Z = x}\, \xi(\pi/2-u,x)\, du\, dx \\
=  \int_{\R^2}
\Big[
\int_\R e^{i \frac{\sin u}{2}\, Z^2} e^{-i \check Z Z} (e^{i\lambda (\cdot)})\, dZ\Big]_{\check Z = x}\, \widetilde \xi(u,x)\, du\, dx,
\end{multline}
where $\widetilde \xi(u,x) = - \sqrt{\cos u}\, \xi(\pi/2-u,\sqrt{\cos u}\, x)$ is a test-function with support about $(0,x_0)$.
The regularized integral is then
\begin{multline*}
\lim\limits_{\varepsilon \rightarrow 0_+}
\int_{\R^2}
\Big[
\int_\R e^{-\varepsilon Z^2}\, e^{i \frac{\sin u}{2}\, Z^2} e^{-i \check Z Z} (e^{i\lambda (\cdot)})\, dZ\Big]_{\check Z = x}\, \widetilde \xi(u,x)\, du\, dx \\
= \lim\limits_{\varepsilon \rightarrow 0_+}
\int_{\R^2}
\Big[
\int_\R e^{-\varepsilon Z^2}
e^{i \frac{\sin u}{2}\, \big(Z^2 -2(\check Z -\lambda)/\sin u\big)} dZ\Big]_{\check Z=x} \, \widetilde \xi(u,x)\, du\, dx \\
= \int_{\R^2}
\Big[\exp \Big( \frac{2i}{\sin u}(\check Z -\lambda)^2\Big))\Big]_{\check Z=x}\,
\sqrt{\frac{2 i\pi}{\sin u}}\, \widetilde \xi(u,x)\, du\, dx \\
=
\int_{\R^2}
\Big[\exp \Big( \frac{i}{v}(\check Z -\lambda)^2\Big))\Big]_{\check Z=x}
\sqrt{\frac{1}{v}}\, \widetilde \theta(v,x)\, dv\, dx
\end{multline*}
for some test-function $(v,x)\mapsto \widetilde \theta(v,x)$ with support about $(0,x_0)$
(one uses here Lebesgue domination theorem and the change of variables $(\sin u)/2\longleftrightarrow v$ about $u=0$).
Though such expression makes sense when $\lambda \in \R$
(since $|\exp \big( i (x- \lambda)^2/v)\big)| =1$ for any point $(v,x)\in {\rm Supp}(\widetilde \xi)$), it does not make sense anymore when $\lambda \in \C$.
In order to overcome this difficulty, one needs to formulate the following lemma.

\begin{lemma}\label{sect6-lem1}
Let $\D(\check Z)$ ($\check Z\in \C$) be a differential operator of the form
\begin{equation}\label{sect6-eq2}
\sum\limits_{\kappa = 0}^\infty \Big[\frac{A_\kappa(\check Z,(d/dZ))}{\kappa!}(\cdot)\Big]_{Z=0}\, (d/dv)^\kappa,
\end{equation}
(where $A_\kappa \in \C[[\check Z,d/dZ]]$ for any $\kappa \in \N$),
considered as acting from the space of entire functions of the variable $Z$ to the space $\C[\check Z][[d/dv]]$.
Suppose that there exist $p\geq 1$ and $\check p\geq 1$ and
$B,\check B\geq 0$ such that
\begin{equation}\label{sect6-eq3}
\sup\limits_{\kappa \in \N,\check Z\in \C}
\big(|A_\kappa(\check Z,W)|\, \exp( - B\, |W|^p - \check B\, |W|^{\check p})\big)<+\infty.
\end{equation}
Then, for any $b\geq 0$, there exists $A^{(b)}\geq 0$ such that
$$
\forall\, C\geq 0,\quad
\forall\, f\in  A_1^{C,b}(\C), \quad
\sup\limits_{\kappa\in \N}
\big|A_\kappa (\check Z,(d/dZ))(f)(0)\big|
\leq C\, A^{(b)}\, e^{\check B\, |\check Z|^{\check p}}.
$$
In particular, for any $f\in A_1^{C,b}(\C)$,
$\D(\check Z)(f)$ remains an infinite order differential operator
$\sum_{\kappa\geq 0} \alpha_\kappa(\check Z)(f)\, (d/dv)^\kappa$ with coefficients satisfying (independently of $f\in A_1^{C,b}(\C)$)
$$
\sum\limits_{\kappa \in \N}
k!\, |\alpha_\kappa (\check Z)(f)|\,
\exp(-B |\check Z|^{\check p}) = C\, A^{(b)} < +\infty.
$$
\end{lemma}

\begin{proof}
The coefficients of $A_\kappa$ as a polynomial in $d/dZ$ satisfy
$$
\sum\limits_{\kappa,j\in \N,\ \check Z\in \C}
|a_{\kappa,j}(\check Z)| \leq C_0 \frac{b_0^j}{\Gamma(j/p)+1} e^{\check B |\check Z|^{\check p}}
$$
for some absolute constants $C_0$ and $b_0$ (Lemma \ref{sect2-lem1}).
As in the proof of Lemma \ref{sect2-lem2}, one concludes that for any $f\in A_1^{C,b}(\C)$ and any $\kappa \in \N$, one has uniform estimates
$|A_\kappa(\check Z,d/dZ)(f)|\leq C\, A^{(b)}\exp (b_0 b|Z|+ \check B
|\check Z|^{\check p})$ for some positive constant $A^{(b)}$.
One gets the required estimates when evaluating at $Z=0$.
\end{proof}

\noindent
One can then complete Proposition \ref{sect5-prop2} into the following companion proposition.

\begin{proposition}\label{sect6-prop1}
Let $\mathscr T = ]0,+\infty[\times \R$ and $\mathscr H~: x\in \R \mapsto -(\partial^2/\partial x^2 + x^2)/2$.
For any $\lambda \in \R$, let $\varphi_\lambda
\in \mathscr D'(\mathscr T,\C)$ be the evolved distribution
from the initial datum $x\in \R \mapsto e^{i\lambda x}$ through the
Cauchy-Kovalewski Schr\"odinger equation \eqref{sect5-eq6}.
Let $\mathscr F=\{\varphi_\lambda\,;\, \lambda \in \R\}$, where
each $\varphi_\lambda$ is considered as a hyperfunction in $\mathscr T$.
Then, for any $a\in \R\setminus [-1,1]$, the sequence
$\big\{\sum_{j=0}^N C_j(N)\, \varphi_{1-2j/N}\big\}_{N\geq 1}$
is a $\mathscr F$-supershift of hyperfunctions over the $\mathscr F$-supershift domain $\mathscr T$.
\end{proposition}

\begin{proof} Let $\theta \in \mathscr D(\R^2_{t,x},\C)$ with support
a small neighborhood $V$ of the point $(\pi/2,x_0)$ ($x_0\in \R$) and
$\widetilde \theta$ the test-function with support $V-(\pi/2,0) \ni (0,x_0)$ that corresponds to it through the successive transformations explicited previously. One has for any $\lambda \in \R$,
\begin{multline}\label{sect6-eq4}
\langle \varphi_\lambda,\theta
\rangle = \int_\R \int_0^\infty
\Big[\Big(\exp \Big( \frac{i}{v}(\check Z -\lambda)^2\Big)\Big]_{\check Z = x}\, \frac{\widetilde \theta(v,x)}{\sqrt v}\, dv\, dx \\
- i \int_\R \int_0^\infty  \Big[\Big(\exp \Big( -\frac{i}{v}(\check Z -\lambda)^2\Big)\Big]_{\check Z = x}\, \frac{\widetilde \theta(-v,x)}{\sqrt v}\, dv\, dx \\
 \\
= \int_\R \int_0^\infty
\Big(
\Big[\sum\limits_{\kappa =0}^\infty
\frac{i^\kappa}{\kappa!}\, \frac{(\check Z- \lambda)^{2\kappa}}{v^{1/2 + \kappa}}\Big]_{\check Z =x}
\, \widetilde \theta(v,x) - i
\Big[\sum\limits_{\kappa =0}^\infty
\frac{(-i)^\kappa}{\kappa!}\, \frac{(\check Z- \lambda)^{2\kappa}}{v^{1/2 + \kappa}}\Big]_{\check Z =x}
\, \widetilde \theta(-v,x)\Big)\, dv\, dx.
\end{multline}
For any $\kappa \in \N$, the distribution $v_+^{-1/2-\kappa}
\in \mathscr D'([0,+\infty[,\R)$ can be expressed as
$$
v_+^{-1/2-\kappa} = \frac{2^\kappa}{\prod_{\ell=1}^\kappa
\big( 2(\kappa -\ell) +1\big)}\, (-d/dv)^\kappa (v_+^{-1/2})
$$
in the sense of distributions in $\mathscr D'([0,+\infty[,\R)$.
Then, one can reformulate  formally \eqref{sect6-eq2} as
\begin{multline}\label{sect6-eq5}
\langle \varphi_\lambda,\theta \rangle
=
\sum\limits_{\kappa = 0}^\infty
\frac{(2i)^\kappa}{\kappa! \prod_{\ell=1}^\kappa
\big( 2(\kappa -\ell) +1\big)} \\
\int_{\R}
\Big\langle
\Big[\Big(\check Z + i\frac{d}{dZ}\Big)^{2\kappa}(e^{i\lambda (\cdot)})\Big]_{\check Z =x}(0)\,
\Big(\frac{d}{dv}\Big)^{\kappa} (
v_+^{-1/2}),\, \widetilde \theta(\cdot,x) - i(-1)^\kappa
\widetilde \theta(-\cdot,x) \Big\rangle\, dx.
\end{multline}
Lemma \ref{sect6-lem1} applies to the two operators
\begin{equation}
\begin{split}
& \D(\check Z) = \sum\limits_{\kappa =0}^\infty
\frac{1}{\kappa!}
\, \Big[
\frac{(2i)^\kappa \, (\check Z + i d/dZ)^{2\kappa}}
{\prod_{\ell=1}^\kappa
\big( 2(\kappa -\ell) +1\big)}(\cdot)\Big]_{Z=0}\, (d/dv)^\kappa \\
& \widetilde \D(\check Z) =
\sum\limits_{\kappa =0}^\infty
\frac{1}{\kappa!}
\, \Big[
\frac{(-i)^{\kappa+1} 2^\kappa \, (\check Z + i d/dZ)^{2\kappa}}
{2 \prod_{\ell=1}^\kappa
\big( (\kappa -\ell) +1\big)}(\cdot)\Big]_{Z=0}\, (d/dv)^\kappa
\end{split}
\end{equation}
with $p=\check p =2$.
These two operators act then continuously (locally uniformly with respect to the parameter $\check Z$) from $A_1(\C)$ into the space of infinite order differential operators in $d/dv$ (depending on the parameter $\check Z\in \C$). Such differential operators can be considered as hyperfunctions on $\R_v$ (elements of $\mathcal H(\R_v)$). Since $v_+^{-1/2}$ is a Fourier hyperfunction in the real line $\R$, the two $\mathcal H(\R)$-valued operators $f\in A_1(\C)
\longmapsto \D(\check Z)(f) \odot v_+^{-1/2}$ and
$f\in A_1(\C) \longmapsto \widetilde \D(\check Z)(f) \odot v_+^{-1/2}$
are well defined (see \cite[Proposition 8.4.8 and Exercise 8.4.5]{kanbook}) and
depend continuously (locally uniformly with respect to $\check Z$)
on the entry $f$ in $A_1(\C)$. Proposition \ref{sect6-prop1} follows then from Theorem \ref{sect2-thm1} and from the expression \eqref{sect6-eq4} (together with its formal reformulation \eqref{sect6-eq5}) for the evaluations $\langle \varphi_\lambda,\theta\rangle$ when $\lambda \in \R$ and $\varphi_\lambda$ is considered as an element in $\mathscr D'(\mathscr T,\C)$ (acting on $\theta \in \mathscr D(\mathscr T,\C)$) which can be also interpreted an a hyperfunction on $\mathscr T$.

\end{proof}

\end{document}